\documentclass[12pt]{amsart}
\usepackage[left=2cm,right=2cm,top=2cm,bottom=2cm]{geometry}
\usepackage{soul}
\usepackage{amsfonts,amsmath,amssymb,amsthm}
\usepackage{dsfont,verbatim}
\usepackage{subcaption}
\usepackage{url}

\usepackage{pgfplots}
\usepackage{hyperref}

\newtheorem{theorem}{Theorem}
\newtheorem{corollary}[theorem]{Corollary}
\newtheorem{lemma}[theorem]{Lemma}

\newtheorem{remark}[theorem]{Remark}

\newtheorem{fact}[theorem]{Fact}

\theoremstyle{definition}
\newtheorem{definition}[theorem]{Definition}
\newtheorem{example}[theorem]{Example}

\newcommand{\ee}{\varepsilon}

\newtheorem{rem}{Remark}

\newcommand{\EE}{\mathbb{E}}

\newcommand{\PP}{\mathbb{P}}
\newcommand{\RR}{\mathbb{R}}

\newcommand{\Scal}{\mathcal{S}}

\newcommand{\dint}{\mathrm{d}}

\begin{document}

\title{Minimax Rates for High-Dimensional Random Tessellation Forests}
\author{Eliza O'Reilly}
\address{Computing and Mathematical Sciences Department, California Institute of Technology, Pasadena, CA 91107}
\email{eoreilly@caltech.edu}

\author{Ngoc Mai Tran}
\address{Department of Mathematics, University of Texas at Austin, Austin, TX 78712}
\email{ntran@math.utexas.edu}

\date{}

\subjclass[2000]{Primary  60D05; Secondary 62G07}

\maketitle
\begin{abstract}%   <- trailing '%' for backward compatibility of .sty file
Random forests are a popular class of algorithms used for regression and classification. The algorithm introduced by Breiman in 2001 and many of its variants are ensembles of randomized decision trees built from {\it axis-aligned} partitions of the feature space. One such variant, called Mondrian forests, was proposed to handle the online setting and is the first class of random forests for which minimax rates were obtained in arbitrary dimension. However, the restriction to axis-aligned splits fails to capture dependencies between features, and random forests that use oblique splits have shown improved empirical performance for many tasks. This work shows that a large class of random forests with general split directions also achieve minimax optimal convergence rates in arbitrary dimension. This class includes STIT forests, a generalization of Mondrian forests to arbitrary split directions, and random forests derived from Poisson hyperplane tessellations. 
These are the first results showing that random forest variants with oblique splits can obtain minimax optimality in arbitrary dimension.
Our proof technique relies on the novel application of the theory of stationary random tessellations in stochastic geometry to statistical learning theory. 
\end{abstract}

\section{Introduction}
Random forests are ensembles of randomized decision trees popularized by Breiman \cite{breiman2001random} and are broadly applicable in classification and regression tasks \cite{fernandez2014we,chen2012random}. Despite their empirical success and widespread use, statistical learning theorems have been notoriously difficult to obtain in dimensions $D \geq 2$ \cite{biau2016random}. There has been significant progress towards understanding the asymptotic behavior of Breiman's original algorithm \cite{ScornetBiauVert2015, WagerAthey2015, WagerAthey2018, MentchHooker2016, ChietalnewRates2022, KlusowskiTian2022}, but much of the theory is still limited by strong assumptions and suboptimal rates. Some of the main difficulties in obtaining theoretical guarantees for this algorithm and its variants result from the complex dependence between the partitioning process generating the tree and the underlying dataset. In response, another line of research considers simplified and stylized versions of random forests. In particular, purely random forests \cite{biau2008consistency,arlot2014analysis} are models built from randomized hierarchical partitions of the input space that are independent of the data and are thus more amenable to theoretical analysis.  
Recently, Mourtada, Gaïffas, and Scornet \cite{mourtada2020minimax} obtained the first minimax optimal rates in arbitrary dimension for a particular class of purely random forests. Specifically, they proved that \emph{Mondrian forests} \cite{lakshminarayanan2014mondrian, lakshminarayanan2016mondrian} attain optimal rates for a properly tuned complexity parameter growing with the amount of data. 

Mondrian forests are purely random forests based on the Mondrian process, a recursive random partition of $\RR^D$ by axis-aligned cuts introduced by Roy and Teh \cite{roy2008mondrian}. This stochastic process enjoys an efficient Markov construction and the following self-consistency property: a sample of a Mondrian process in some domain $W_1 \subset \RR^D$ has the same distribution as sampling a Mondrian process on a larger domain $W_2$ such that $W_1 \subset W_2$ and intersecting with $W_1$. These properties ensure the amenability of Mondrian forests to the online setting \cite{lakshminarayanan2014mondrian,lakshminarayanan2016mondrian,wang2018batched}, where data arrives in a streaming manner, and the estimator is updated over time. This contrasts with Breiman's random forest algorithm and many of its variants which are restricted to the batch setting, where the entire dataset is used at once to build the model. In addition to this practical advantage, the results of \cite{mourtada2020minimax} mentioned above highlight the theoretical advantages of Mondrian forests resulting from its construction.

One key limitation of the Mondrian forest and Breiman's original random forest algorithm comes from the constraint that only one feature of the input is used each time the data within a tree node is divided. While computationally efficient, these axis-aligned splits cannot capture dependencies between features and may produce complex step-wise decision boundaries that lead to high variance and overfitting. In practice, this places a large burden on the feature selection and representation process. To address these concerns, Breiman \cite{breiman2001random} proposed a variant called Forest-RC that increases the expressiveness of the model by allowing splits using linear combinations of features, and it was shown to achieve improved empirical performance over the axis-aligned version. Many other models of random forests using oblique splits have subsequently been proposed \cite{Menzeetal2011, blaser2016random, pmlr-v84-fan18b, TehRTFs2019, rainforth2015canonical}. To mitigate the increased computational cost of using linear combinations of features, recent work \cite{Tomita2020} studied random forests with oblique splits from sparse projections using only a small subset of features. However, oblique splits increase the already difficult task of proving theoretical guarantees for random forest algorithms, and theory justifying the empirical performance of these variants is extremely limited. Existing guarantees restricted to the setting of axis-aligned splits are not easily generalized to oblique splits, illuminating a need for a more flexible theoretical framework. 

To the best of our knowledge, this paper gives the first results on minimax optimality for a large class of purely random forests that are defined for all dimensions and allow for {\it general} splits using linear combinations of features. In particular, we show that STIT forests, a significant generalization of Mondrian forests, also attain the minimax optimal rates proved in \cite{mourtada2020minimax}. STIT forests are derived from the stable under iteration (STIT) processes introduced by Nagel and Weiss \cite{Nagel2003, Nagel2005}, and these stochastic processes all enjoy the self-consistency and online construction that underpin the popularity of the Mondrian forest in practice. The family of STIT processes is indexed by probability distributions on the unit sphere describing the distribution of directions of the hyperplane cuts in the random partition, and the Mondrian process corresponds to a STIT process where this directional distribution is the discrete uniform measure on the coordinate vectors. Subsequent generalizations of the Mondrian process to oblique cuts \cite{pmlr-v84-fan18b,TehRTFs2019} are also special cases of STIT processes. The freedom in the choice of the directional distribution for the splits brings greater flexibility in building machine learning models. For example, while the Mondrian process can only be used to approximate the Laplace kernel \cite{lakshminarayanan2016mondrian}, STIT processes produce random features that approximate a much broader class of kernels \cite{OReillyTran2021}. Improved empirical performance of STIT forests built from STIT processes with a uniform directional distribution over Mondrian forests was also shown in \cite{TehRTFs2019} through a classification task and simulation study. Additionally, a STIT process in $\RR^D$ with discrete directional distribution having $N \geq D$ support vectors can be simulated by lifting to a subspace of $\RR^N$ and running a Mondrian process \cite{OReillyTran2021}. This observation can mitigate computational costs of the oblique splits and gives an interpretation of a STIT forest as implicitly generating a Mondrian forest in a higher dimensional feature space.

A significant contribution of our work lies in the proof technique, as our approach relies on theorems in stochastic geometry that have not previously been utilized in statistical learning theory. In extending the theory for the Mondrian to STIT forests, a fundamental difficulty is the geometry of the cells of the partitions. The Mondrian process generates axis-aligned rectangular cells, and the distribution of the cell a given input is contained in can be characterized precisely from the construction. In contrast, STIT processes divide the input space into more general and complex convex polytopes. The theory of random tessellations in stochastic geometry provides a flexible and robust theoretical framework that enables us to handle these more general cell geometries. In particular, we crucially exploit the self-consistency property and the stationarity of the corresponding STIT tessellation on $\RR^D$ to obtain risk bounds for STIT forest estimators that depend on the distribution of a single random polytope, called the {\it typical cell} of the random tessellation. 

Additionally, our proof technique allows us to incorporate an assumption of intrinsic low-dimensionality on the input data, improving convergence rates in high-dimensional feature space. A well-known challenge in developing statistical learning guarantees for nonparametric regression is the curse of dimensionality, where, for instance, one needs $O(\ee^{-D})$ number of samples to estimate general Lipschitz functions on $\RR^D$ with $\ee$ accuracy. Indeed, the minimax optimal rates for Mondrian forests in \cite{mourtada2020minimax} depend on the ambient dimension of the input data and become very slow in the presence of high-dimensional feature space, even though empirically random forests perform well in such regimes \cite{chen2012random,fernandez2014we,Tomita2020}. One approach to justifying such performance is to make additional structural assumptions on the input data source to attain improved convergence rates. %, as in the results of \cite{KPOTUFE2012}. 
For STIT forests, the theory of stationary hyperplane processses in stochastic geometry provides insight yielding optimal rates that also adapt to a notion of intrinsic dimensionality of the input.

Specifically, our first main result (see Theorem \ref{t:rate1}) gives an upper bound on the quadratic risk of a STIT forest regression estimator of a $\beta$-H\"{o}lder continuous function for $\beta \in (0,1]$. The theorem implies that \emph{any} STIT forest with optimally tuned complexity parameter achieves the minimax rate for this function class, and the choice of the directional distribution of the splits appears in the constant terms of the upper bound. Our proof method gives geometric interpretations to these constants in terms of moments of the diameter of the cell of the associated STIT tessellation containing the origin and the expected mixed volumes between the support of the input and the typical cell. This precise geometry enables us to  obtain rates in terms of the intrinsic dimension of the input defined by the dimension of the subspace of $\RR^D$ on which the input data is supported. In the particular case of the Mondrian forest, the typical cell is the Minkowski sum of i.i.d. centered line segments parallel to the axes with exponential length. Taking the support of the input to be $[0,1]^D$ recovers Theorem 2 in \cite{mourtada2020minimax}, see Example \ref{ex:Nlambda_ex3} for details. Our second main result (see Theorem \ref{t:rate2}) significantly generalizes Theorem 3 in \cite{mourtada2020minimax}, where additional smoothness assumptions are made on $f$. We show that any STIT forest estimator achieves the minimax rate for the class of $(1 + \beta)$-H\"{o}lder functions for $\beta \in (0,1]$ for both a large enough number of trees in the forest and an optimally tuned complexity parameter. As was the case for Mondrian forests, an improved rate for STIT forests over STIT trees is due to large enough forests having a smaller bias than single trees for smooth regression functions.

Finally, our proof technique also takes us \emph{beyond} the class of STIT forests. Any random partition of the input space can be used to define a random tree estimator and subsequently a random forest estimator. Since the upper bounds we obtain on the convergence rates are explicitly derived in terms of geometric properties of the typical cell of the random partition, it can readily be applied to any random forest obtained from a stationary random tessellation of $\RR^D$. Our last main result (see Theorem \ref{t:poisson}) demonstrates this principle. It states that a random forest derived from a Poisson hyperplane process achieves identical convergence rates as a STIT forest with the same directional distribution and complexity parameter, and thus is also minimax optimal. 

\subsection*{Organization} Section \ref{sec:background} collects background on STIT processes, Poisson hyperplane processes, and essential results in stochastic geometry needed for our proofs. Section \ref{sec:cells} presents key lemmas on distributional characteristics of the cells of STIT and Poisson hyperplane tessellations. Section \ref{s:main} then states our three main results, Theorems \ref{t:rate1},  \ref{t:rate2}, and \ref{t:poisson}, and in \ref{sec:classification}, risk bounds in the setting of binary classification are obtained as corollaries. Section \ref{sec:proofs} provides the proofs of these results.
Section \ref{s:discussion} concludes with discussions and open problems. 
\subsection*{Acknowledgements} Ngoc Mai Tran is supported by NSF Grant DMS-2113468 and the NSF IFML 2019844 award to the University of Texas at Austin. Eliza O'Reilly is supported by NSF MSPRF Award 2002255 with additional funding from ONR Award N00014-18-1-2363. 

\section{Preliminaries}\label{sec:background}

We recall here the key concepts from stochastic geometry needed for our paper, including \emph{random tessellations, stationarity}, the \emph{zero cell}, and the \emph{typical cell}. We recommend the book by \cite[Chapter~10]{weil} for additional background.  

A tessellation is a locally finite random partition of $\RR^D$ into compact and convex polytopes. It can be viewed as the collection of polytopes, or cells, of the tessellation or as the union of their boundaries. In this paper, we view tessellations as the collection of cells, but we will also discuss the union of cell boundaries to establish relevant definitions. Formally, we define a \emph{random tessellation} as a point process of cells $\mathcal{P} = \{C_i\}_{i \in \mathbb{Z}}$ taking values in the space $\mathcal{K}$ of non-empty compact and convex polytopes that satisfy the following:
\begin{itemize}
    \item for all compact $K \subset \RR^D$, a finite number of $C_i$'s have non-empty intersection with $K$;
    \item for all $i \neq j$, $\mathrm{int}(C_i) \cap \mathrm{int}(C_j) = \emptyset$;
    \item $\underset{i \in \mathbb{Z}}{\bigcup} C_i = \RR^D$.
\end{itemize}
A random tessellation $\mathcal{P} = \{C_i\}_{i \in \mathbb{Z}}$ is \emph{stationary} if its distribution is invariant under translations. That is, for all $x \in \RR^D$, $\{C_i + x\}_{i \in \mathbb{Z}} \stackrel{d}{=} \{C_i\}_{i\in \mathbb{Z}}$, where `$\stackrel{d}{=}$' denotes equality in distribution. A consequence of the definition of a random tessellation is that for all $i$, $\mathrm{vol}_D(\partial C_i) = 0$. This fact along with stationarity implies that every $x \in \RR^D$ almost surely belongs to a unique cell of the tessellation, which will be denoted $Z_x$. The zero cell of $\mathcal{P}$, denoted by $Z_0$, is defined as the unique cell of the tessellation containing the origin. Stationary also implies that $Z_0 \stackrel{d}{=} Z_x - x$.

An important random object related to a stationary random tessellation $\mathcal{P}$ is the {\it typical cell}. To define this, consider first a center function $c: \mathcal{K} \to \RR^D$ such that $c(K + x) = c(K) + x$ for all $x \in \RR^D$. Examples include the centroid or the center of the smallest ball containing $K$. We can then decompose $\mathcal{P}$ into a stationary marked point process $\{\left(c(C_j), C_j - c(C_j)\right)\}_{j \in \mathbb{Z}}$ consisting of a ground point process in $\RR^D$ of cell centers and elements from $\mathcal{K}_0 := \{ K \in \mathcal{K} : c(K) = 0\}$ attached to each center. Following \cite[Section~4.1]{weil}, there exists a random polytope $Z$ in $\mathcal{K}_0$ such that 
for any non-negative measurable function $f$ on $\mathcal{K}$,
\begin{align}\label{e:campbell}
    \EE\left[\sum_{C \in \mathcal{P}} f(C)\right] = \frac{1}{\EE[\mathrm{vol}_D(Z)]}\EE\left[\int_{\RR^D} f(Z + y) \dint y\right].%,
\end{align}
The random polytope $Z$ is called the \emph{typical cell} of $\mathcal{P}$.
Its distribution can be understood as the limiting distribution of a cell chosen uniformly at random from a large ball, centered at the origin using the center function $c$, as the radius of the ball grows to infinity. The relation \eqref{e:campbell} also implies that the distribution of the centered zero cell $Z_0 - c(Z_0)$ has the same distribution as the volume weighted typical cell, i.e.
\begin{align*}
    \EE[f(Z_0 - c(Z_0))] 
    &= \frac{1}{\EE[\mathrm{vol}_D(Z)]}\EE\left[\mathrm{vol}_D(Z) f(Z)\right].
\end{align*}

For a random tessellation $\mathcal{P}$ in $\RR^D$, we will denote by $\mathcal{Y}$ the union of cell boundaries, which forms a $d-1$-surface process in $\RR^D$ \cite[Section 4.5]{weil}. For the stationary surface process $\mathcal{Y}$ corresponding to a stationary random tessellation, we can define a directional distribution $\phi$ on $\mathbb{S}^{D-1}$ characterizing the `rose of directions' for the $D-1$-dimensional facets generating the cell boundaries. We will say $\mathcal{P}$ has directional distribution $\phi$ if $\mathcal{Y}$ has directional distribution $\phi$.

\subsection{STIT Tessellations}

For two random tessellations $\mathcal{P}_1$ and $\mathcal{P}_2$, denote the union of their cell boundaries by $\mathcal{Y}_1$ and $\mathcal{Y}_2$, respectively. Associate to each cell $c$ in $\mathcal{P}_1$ an independent copy $\mathcal{Y}_2(c)$ of $\mathcal{Y}_2$ and assume the family $\{\mathcal{Y}_2(c) : c \in \mathcal{P}_1\}$ is independent of $\mathcal{P}_1$. Then, the {\it iteration} of $\mathcal{Y}_1$ and $\mathcal{Y}_2$ is defined as
\begin{align*}
    \mathcal{Y}_1 \boxplus \mathcal{Y}_2 := \mathcal{Y}_1 \cup \bigcup_{c \in \mathcal{P}_1} \left(\mathcal{Y}_2(c) \cap c\right).
\end{align*}
That is, each cell $c$ of the frame tessellation $\mathcal{P}_1$ is subdivided by the cells of $\mathcal{Y}_2(c) \cap c$.
A random tessellation $\mathcal{P}$ is called \textit{stable under iteration}, or STIT, if for the union of cell boundaries $\mathcal{Y}$, for all $n \in \mathbb{N}$,
\begin{align}\label{e:STIT}
 \mathcal{Y} \stackrel{d}{=} n (\underbrace{\mathcal{Y} \boxplus \cdots \boxplus \mathcal{Y}}_{n \text{ times}}),   
\end{align}
where $n\mathcal{Y} := \{nx : x \in \mathcal{Y}\}$ is the dilation of $\mathcal{Y}$ by the factor $n$. 

A STIT process is a stochastic process $\{\mathcal{Y}(\lambda): \lambda > 0\}$ of random tessellation cell boundaries in $\RR^D$ with the following properties:
\begin{itemize}
    \item[(i)] Stationarity: $\mathcal{Y}(\lambda) + x \stackrel{d}{=} \mathcal{Y}(\lambda)$ for all $x \in \RR^D$;
    \item[(ii)] Markov Property: $\mathcal{Y}(\lambda_1 + \lambda_2) \stackrel{d}{=} \mathcal{Y}(\lambda_1) \boxplus \mathcal{Y}(\lambda_2)$ for all $\lambda_1,\lambda_2 > 0$;
    \item[(iii)] STIT: for all $\lambda > 0$ and $n \in \mathbb{N}$, \eqref{e:STIT} holds for $\mathcal{Y}(\lambda)$. 
\end{itemize}

We will call the parameter $\lambda > 0$ the \emph{lifetime} of the process. A consequence of property (iii) is that STIT processes have the following scaling property: for all $\lambda > 0$, $\mathcal{Y}(1) \stackrel{d}{=} \lambda \mathcal{Y}(\lambda)$ \cite[Lemma 5]{Nagel2005}. Intuitively, this property says that one can swap time for space. If we fix a compact observation window $W \subset \RR^D$, then $\{\mathcal{Y}(\lambda) \cap W : \lambda > 0\}$ is a stochastic process of random tessellation cell boundaries in $W$, which we think of as a visualization of $\mathcal{Y}(\lambda)$ through the window $W$. Now, one can fix $\lambda$ and `zoom out' on $\mathcal{Y}(\lambda)$ by mapping $\mathcal{Y}(\lambda) \cap W \mapsto \frac{1}{2}\mathcal{Y}(\lambda) \cap W$, or one can run the STIT process for twice as long by mapping $\mathcal{Y}(\lambda) \cap W \mapsto \mathcal{Y}(2\lambda) \cap W$. The scaling property says these two operations give the same random tessellation in $W$ in distribution.  

For a lifetime $\lambda > 0$, let $\mathcal{P}(\lambda)$ denote the STIT tessellation of $\RR^D$ with cell boundaries given by a STIT process $\mathcal{Y}(\lambda)$. Denote the zero cell of $\mathcal{P}(\lambda)$ by $Z_0^{\lambda}$ and the typical cell by $Z_{\lambda}$. The scaling property implies the following important facts that we will use in the remainder of the paper: for all $\lambda > 0$,
\begin{align}\label{e:zerocell_scaled}
Z_0 \stackrel{d}{=} \lambda Z_0^{\lambda} \quad \text{ and } \quad  Z \stackrel{d}{=} \lambda Z_\lambda,
    \end{align}
where $Z_0 := Z_0^1$ and $Z := Z_1$ are the zero cell and typical cell of $\mathcal{P}(1)$.

While seemingly abstract, it was proved in \cite{Nagel2005} that the local STIT process $\{\mathcal{Y}(\lambda) \cap W: \lambda > 0\}$ restricted to a fixed compact and convex window $W \subset \RR^D$ can be simulated through a Markov process that generates a hierarchical partition of $W$ over time. A special case of this construction was rediscovered in \cite{roy2008mondrian}, which led to the Mondrian process.  

Formally, let $\phi$ be an even probability measure on the unit sphere $\mathbb{S}^{D-1}$ with support containing $d$ linearly independent directions. Then define $\Lambda$ to be the stationary and locally finite measure on the space of hyperplanes in $\RR^{D}$, denoted $\mathcal{H}^{D}$, such that
\[\Lambda(A) := \int_{\RR} \int_{\mathbb{S}^{D-1}} 1_{\{H(u,t) \in A\}} \phi(\dint u)\dint t , \quad A \in \mathcal{B}(\mathcal{H}^{D}),\]
where $H(u,t) := \{x \in \RR^D : \langle x, u \rangle = t\}$. The space $\mathcal{H}^{D}$ of hyperplanes is equipped with the hit-miss topology, which contains compact subsets of the following form: for compact $W \subset \RR^D$, %define
\[[W] := \{H \in \mathcal{H}^D : H \cap W \neq \emptyset\}.\]
Now, fix a lifetime $\lambda > 0$ and consider the following procedure to construct a random partition $\mathcal{Y}(\lambda,W,\phi)$ of $W$.
\begin{enumerate}
    \item Draw $\delta \sim \mathrm{Exp}(\Lambda([W])$, where
\begin{align*}
 \Lambda([W]) &=  \int_{\RR} \int_{\mathbb{S}^{D-1}} 1_{\{H(u,t) \cap W \neq \emptyset\}} \dint \phi(u) \dint t 
= \int_{\mathbb{S}^{D-1}} \left(h(W,u) + h(W,-u)\right) \dint \phi(u),   
\end{align*}
and $h(W, u) := \sup_{x \in W} \langle u,x \rangle$ is the support function of $W$. 
\item If $\delta > \lambda$, stop. Else, at time $\delta$, generate a random hyperplane $H(U,T)$ where the direction $U$ is drawn from the distribution
\[\dint \Phi(u) := \frac{h(W,u) + h(W,-u)}{\Lambda([W])}\dint\phi(u), \quad u \in \mathbb{S}^{D-1},\]
and conditioned on $U$, $T$ is drawn uniformly on the interval from $-h(W,-U)$ to $h(W,U)$. 
Split the window $W$ into two cells $W_1$ and $W_2$ with $H(U,T)$.
\item Repeat steps (1) and (2) in each sub-window $W_1$ and $W_2$ independently with new lifetime parameter $\lambda - \delta$ until the lifetime expires.
\end{enumerate}

In Figure \ref{fig:example}, we show samples of the partition generated from a STIT process in the centered unit square $[-0.5,0.5]^2$ using the above procedure over increasing lifetime $\lambda$. In this example, the directional distribution is uniform over the three directions $e_1:= (1,0)$, $e_2:= (0,1)$, $u := (1/\sqrt{2}, 1/\sqrt{2})$, and their reflections over the origin.

\begin{figure}[h!]
    \centering
\begin{minipage}{.24\linewidth}
  \includegraphics[width=\linewidth]{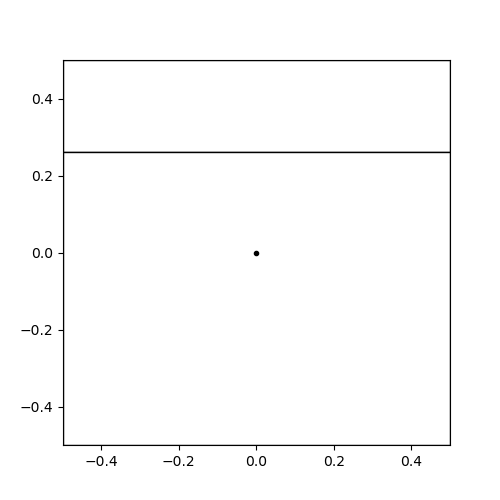}
\end{minipage}
%\hspace{.01\linewidth}
\begin{minipage}{.24\linewidth}
  \includegraphics[width=\linewidth]{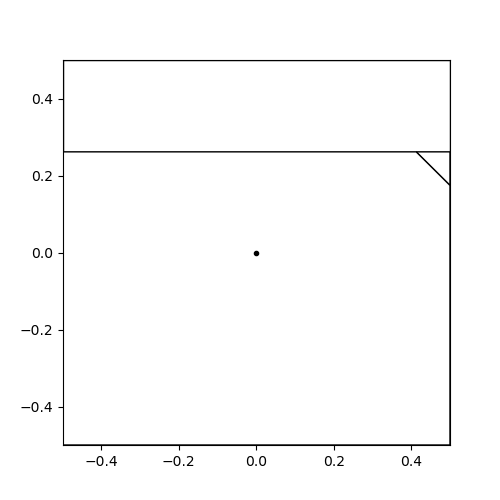}
\end{minipage}
%\hspace{.01\linewidth}
\begin{minipage}{.24\linewidth}
  \includegraphics[width=\linewidth]{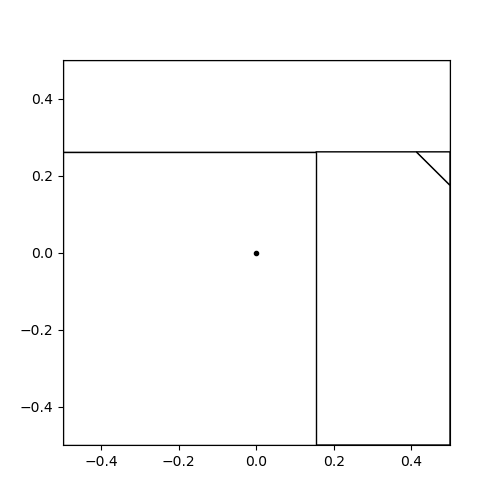}
\end{minipage}
%\hspace{.01\linewidth}
\begin{minipage}{.24\linewidth}
  \includegraphics[width=\linewidth]{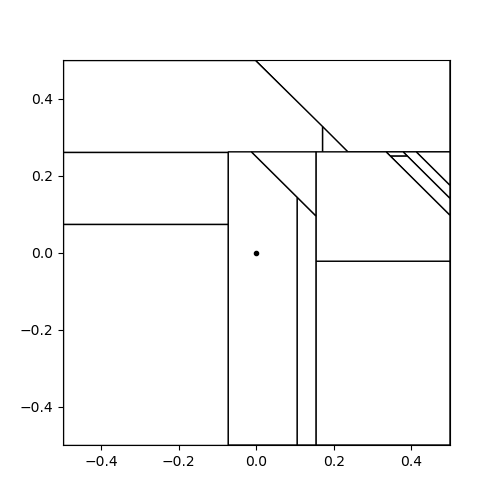}
\end{minipage}
    \caption{An example STIT process with three directions. When the process within a cell $W$ begins, an independent exponential clock with mean $\Lambda([W])$ is started. When the clock rings, the cell is cut by a random hyperplane conditioned to hit this cell as in step $(2)$ above. In this simulation, we start with the unit square $W = [-0.5,0.5]^2$ and run the process until the lifetime $\lambda = 9$. The first three figures show the resulting partition after the first three cuts are made and the last figure shows the final STIT tessellation at time $\lambda = 9$.}
    \label{fig:example}
\end{figure}

Theorem 1 in \cite{Nagel2005} shows the existence of a STIT tessellation process $\mathcal{Y}(\lambda)$ on $\RR^D$ such that $\mathcal{Y}(\lambda) \cap W \stackrel{d}{=} \mathcal{Y}(\lambda, W,\phi)$. Conversely, for any stationary random tessellation in $\RR^D$ with cell boundaries $\mathcal{Y}$ satisfying \eqref{e:STIT} with directional distribution $\phi$, Corollary 2 in \cite{Nagel2005} shows that there exists $\lambda > 0$ such that %$\mathcal{Y} \stackrel{d}{=} \mathcal{Y}(\lambda)$, and 
$\mathcal{Y} \cap W \stackrel{d}{=} \mathcal{Y}(\lambda,W,\phi)$ for all compact $W \subset \RR^D$. 
Together, these results imply that the class of STIT tessellations is the most general class of stationary random tessellations with the hierarchical construction that underpins the computational benefits of the Mondrian process \cite{roy2008mondrian, lakshminarayanan2014mondrian, lakshminarayanan2016mondrian}.

\begin{example}[The Mondrian process as a special case of STIT processes]\label{ex:mondrian}
If $\phi$ is the uniform distribution over the positive and negative standard basis vectors, i.e. 
\begin{align}\label{e:mondrian_phi}
\phi = \frac{1}{2D} \sum_{i=1}^D (\delta_{e_i} + \delta_{-e_i}),    
\end{align}
then the resulting STIT process $\mathcal{Y}(\lambda,W,\phi)$ is the Mondrian process \cite{roy2008mondrian}. See Figure \ref{fig:four.partitions}(B) for a simulation.
\end{example}

\begin{example}[Isotropic STIT process]\label{ex:isotropic}
If $\phi$ is the uniform distribution over $\mathbb{S}^{D-1}$, then the distribution of the corresponding STIT tessellation is invariant with respect to rotations about the origin. This model is called the isotropic STIT process. See Figure \ref{fig:four.partitions}(A) for a simulation. 
\end{example}

\begin{center}
\begin{figure}
\centering
  \begin{subfigure}[t]{.4\textwidth}
    \centering
\includegraphics[width=\textwidth]{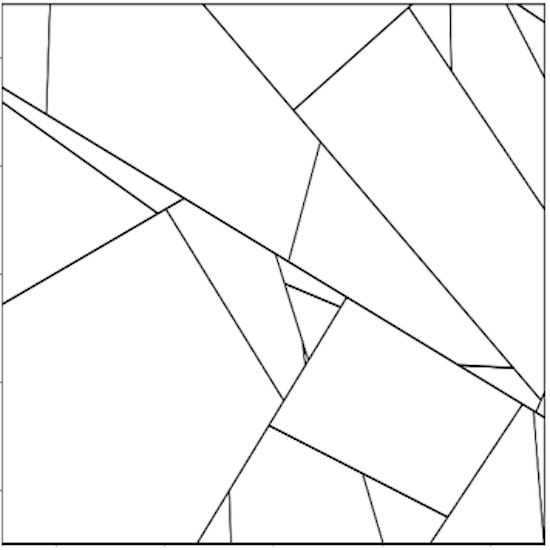}
    \caption{Isotropic STIT process}
  \end{subfigure}
  \qquad \hspace{.1in}%\hfill
  \begin{subfigure}[t]{.4\textwidth}
    \centering
\includegraphics[width=\textwidth]{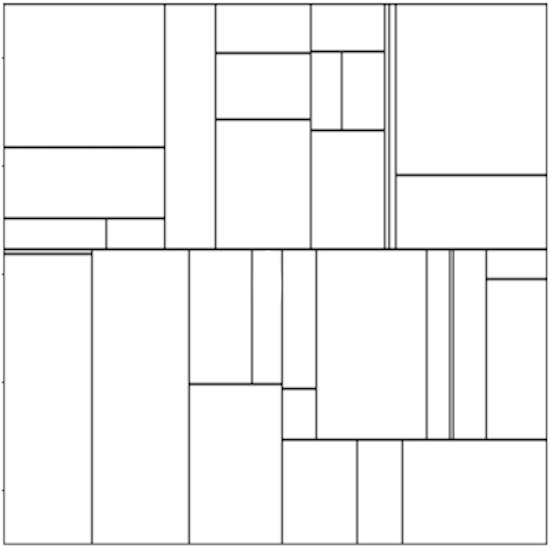}
    \caption{Axis-aligned STIT (aka Mondrian) process}
  \end{subfigure}  
  \medskip \vspace{.1in}
  \begin{subfigure}[t]{.4\textwidth}
  \centering
\includegraphics[width=\textwidth]{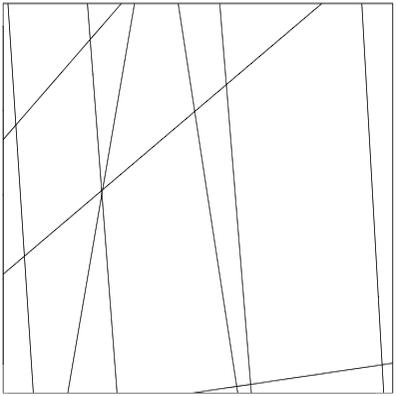}
    \caption{Isotropic Poisson hyperplane process}
  \end{subfigure}
   \qquad \hspace{.1in}%\hfill
  \begin{subfigure}[t]{.4\textwidth}
    \centering
\includegraphics[width=\textwidth]{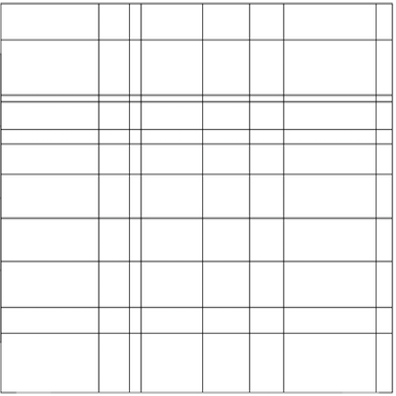}
    \caption{Axis-aligned (aka Manhattan) Poisson hyperplane process}
  \end{subfigure}
  \caption{A simulation of STIT processes (top) vs. Poisson hyperplane processes (bottom) up to time $\lambda = 10$, with $\phi$ being the continuous uniform measure on the unit circle in (A) and (C), and the discrete uniform measure on the standard coordinate vectors in (B) and (D). Though the pair (A,C) (respectively (B,D)) are globally different tessellations, it was observed in \cite{Nagel2005} (see also Corollary 1 in \cite{Thale2013Poisson}) that the typical cell $Z$ of these two tessellations have identical distribution. This fact is key to the proof of Theorem \ref{t:poisson}, which says that the random forests based on (A) and (C) (respectively (B) and (D)) achieve the same minimax rates.  }\label{fig:four.partitions}
\end{figure}
\end{center}

\subsection{Poisson Hyperplane Tessellations}
We now define another class of stationary random tessellations of $\RR^D$, and describe its relationship to the class of STIT tessellations. A \emph{stationary Poisson hyperplane process} $X$ is a stationary Poisson point process on the space of affine hyperplanes $\mathcal{H}^D$ in $\RR^D$ with first moment measure
\[\Theta(\cdot)  := \EE[X(\cdot)] = \lambda\int_{\RR} \int_{\mathbb{S}^{D-1}} 1_{\{H(u,t) \in \cdot\}} \dint \phi(u) \dint t,\]
for some constant $\lambda > 0$ called the \emph{intensity}, and an even probability measure $\phi$ on $\mathbb{S}^{D-1}$ called the spherical directional distribution \cite[Chapter 4.4]{weil}. The following procedure generates a sample from $X$ on a compact window $W$:
\begin{enumerate}
    \item Sample $N \sim \mathrm{Poisson}(\Theta([W]))$, where
    \[\Theta([W]) = \lambda \int_{\RR}\int_{\mathbb{S}^{D-1}} 1_{\{H(u,t) \cap W \neq \emptyset\}} \dint \phi(u) \dint t = \lambda \int_{\mathbb{S}^{D-1}} \left(h(W,u) + h(W,-u)\right) \dint \phi(u).\]
    \item Conditioned on $N = n$, generate $n$ i.i.d. random hyperplanes $\{H(U_i, T_i)\}_{i=1}^n$, where for each $i$, $U_i$ has probability distribution
    \[\dint \Phi(u) := \frac{\lambda(h(W,u) + h(W,-u))}{\Theta([W])}\dint\phi(u),\quad u \in \mathbb{S}^{D-1},\]
    and conditioned on $U_i$, 
    $T_i$ is uniform in the interval from $-h(W, -U_i)$ to $h(W, U_i)$.
\end{enumerate}

Poisson hyperplane processes induce random tessellations on $\RR^D$ called Poisson hyperplane tessellations that are globally different than STIT tessellations (cf. Figure \ref{fig:four.partitions}). In particular, a Poisson hyperplane tessellation is face-to-face, meaning that the intersection of two cells is either empty, or is a face of both cells. This is not the case for a STIT tessellation. For example, a vertex of a cell in a STIT tessellation can be an interior point of the facet of a neighbor cell. However, the typical cell of a STIT tessellation with lifetime parameter $\lambda$ and directional distribution $\phi$ has the same distribution as the typical cell of a stationary Poisson hyperplane tessellation with intensity $\lambda$ and the same directional distribution \cite[Corollary 1]{Thale2013Poisson}. By Theorem 10.4.1 in \cite{weil}, the typical cell determines the distribution of the zero cell, and so the zero cell of a STIT tessellation and a stationary Poisson hyperplane tessellation with corresponding parameters are also equal in distribution.

\section{Parameters of STIT and Poisson Hyperplane Tessellation Cells}\label{sec:cells}

In this section, we generalize the results in Section 4 of \cite{mourtada2020minimax} on the diameter of the zero cell and the number of cells hitting a compact and convex domain. These observations show that STIT processes and stationary Poisson hyperplane processes produce partitions on a domain of unit volume that contain $O(\lambda^d)$ cells of diameter $O(1/\lambda)$ which is on the order of the $1/\lambda$-covering number for such a domain. 
Both bounds depend on an important parameter called the \textit{associated zonoid} \cite[p.156]{weil}. If the random tessellation has directional distribution $\phi$ and lifetime/intensity $\lambda$, this is defined as the convex body $\Pi_{\lambda}$ in $\mathbb{R}^D$ with support function
\begin{align}\label{eq:hZ}
h(\Pi_{\lambda},v) =\frac{\lambda}{2}\Lambda([[0,v]]) = \frac{\lambda}{2}\int_{\mathbb{S}^{D-1}} |\langle u,v \rangle| \dint \phi(u), \quad v \in \mathbb{S}^{D-1}.
\end{align}
We will denote by $\Pi$ the associated zonoid of the process for lifetime/intensity $\lambda = 1$ and refer to this parameter that depends only on the directional distribution as the \emph{normalized associated zonoid} of the random tessellation. In particular, note that $\Pi_{\lambda} = \lambda \Pi$. 
We also recall \cite[(10.4) and (10.44)]{weil} that
\begin{align}\label{e:EVZ}
\EE[\mathrm{vol}_D(Z)] = \frac{1}{\mathrm{vol}_D(\Pi)}.
\end{align}

For the Mondrian (see Example \ref{ex:mondrian}) or axis-aligned Poisson hyperplane process, we can combine \eqref{e:mondrian_phi} and \eqref{eq:hZ} to obtain the support function of $\Pi$:
\begin{align*}
 h(\Pi, v) = \frac{1}{D} \sum_{i=1}^D |\langle e_i, v\rangle| = \frac{1}{D}\|v\|_1, \quad u \in \RR^D.  
\end{align*} 
Thus the associated zonoid is the $\ell^{\infty}$ ball $\Pi = \{x \in \RR^D : \|x\|_{\infty} \leq \frac{1}{D}\}$. 

For the isotropic STIT (see Example \ref{ex:isotropic}) or isotropic Poisson hyperplane process, we let $\phi$ in \eqref{eq:hZ} be the uniform distribution $\sigma$ on the unit sphere to obtain
$$h(\Pi,v) = \frac{1}{2}\int_{\mathbb{S}^{D-1}} |\langle u,v \rangle| \dint \sigma(u) =  \frac{\|v\|_2}{2}\int_{\mathbb{S}^{D-1}} |u_n| \dint \sigma(u) = c_D \|v\|_2,$$ 
where $c_D:= \frac{\Gamma\left(\frac{D}{2}\right)}{2\sqrt{\pi}\Gamma\left(\frac{D+1}{2}\right)}$. Thus the associated zonoid $\Pi$ is an $\ell^2$ ball centered at the origin with radius $c_D$. 

The associated zonoid can also be used to understand the distribution of the intersection of a STIT or Poisson hyperplane tessellation in $\RR^d$ with a linear subspace $\mathcal{S}$. We will repeatedly use the following important fact in the remainder of the paper. 
\begin{fact}{\cite[(4.61)]{weil}}\label{f:intersection}
Let $\mathcal{P}$ be a STIT (or stationary Poisson hyperplane) tessellation in $\RR^D$ with associated zonoid $\Pi$ and let $\mathcal{S}$ be a linear subspace of $\RR^D$. The intersection $\mathcal{P} \cap \mathcal{S}$ is a STIT (or stationary Poisson hyperplane) tessellation in $\mathcal{S}$ with associated zonoid given by the orthogonal projection $P_S\Pi$ of $\Pi$ onto the subspace $\mathcal{S}$.
\end{fact}

In the following, let $\mathcal{P}(\lambda)$ denote a STIT tessellation in $\RR^D$ with lifetime $\lambda \in (0, \infty)$ and normalized associated zonoid $\Pi$.

\subsection{Diameter of zero cell}

The precise distribution of the diameter of the zero cell of $\mathcal{P}(\lambda)$ with a general directional distribution remains an open question in stochastic geometry. However, we will provide an upper bound on the moments that is sufficient for proving this paper's results.

\begin{lemma}\label{l:diameter}
Let $Z^{\lambda}_0$ denote the zero cell of $\mathcal{P}(\lambda)$ in $\RR^D$ and let $\mathcal{S}$ be a linear subspace of $\RR^D$.
Then, for all $k > 0$, there exists a constant $c := c(k, \Pi)$ depending only on $k$ and $\Pi$ such that
\begin{align*}
  \EE[\mathrm{diam}(Z^{\lambda}_0 \cap \mathcal{S})^k] \leq c \lambda^{-k}h_{\min}(P_{\mathcal{S}}\Pi)^{-k},
\end{align*}
where $h_{\min}(P_{\Scal}\Pi):= \min_{u \in \mathbb{S}^{D-1} \cap \mathcal{S}} h(P_{\Scal}\Pi, u)$.
\end{lemma}
\begin{proof}
We first note that by Fact \ref{f:intersection} and \eqref{e:zerocell_scaled}, the random polytope $Z_0^{\lambda} \cap \mathcal{S}$ has the same distribution as $\lambda^{-1}Z_{0, \mathcal{S}}$, where $Z_{0,\mathcal{S}}$ is the zero cell of the STIT tessellation in $\mathcal{S}$ with associated zonoid $P_{\mathcal{S}}\Pi$. \\
Next, it follows from Section 8 of \cite{HugSchneider2007} that for fixed $r > 0$ and $\tau \in (0,1)$, there exists a constant $c_{r} = c_{r}(\tau, \Pi)$ such that for all $a > r$,
\begin{align}\label{e:diamtail}
    \PP(\mathrm{diam}(Z_{0,\mathcal{S}}) \geq a) \leq c_{r} e^{-\tau a h_{\min}(P_{\mathcal{S}}\Pi)}.
\end{align}
The moments of the diameter then satisfy
\begin{align*}
  \EE[\mathrm{diam}(Z_{0, \mathcal{S}})^k] &= \int_{0}^{\infty} kt^{k-1}\PP(\mathrm{diam}(Z_{0,\mathcal{S}})\geq t) \dint t \\
  &\leq \int_{0}^{r} kt^{k-1} \dint t + k c_{r} \int_{r}^{\infty} t^{k-1} e^{-\tau t h_{\min}(P_{\mathcal{S}}\Pi)} \dint t \\
  &= r^k + \frac{k c_r }{(\tau h_{\min}(P_{\mathcal{S}}\Pi))^{k}} \int_{0}^{\infty} y^{k-1}e^{-y} \dint y \leq r^k + \frac{c_{r} \Gamma(k+1)}{\tau^kh_{\min}(P_{\mathcal{S}}\Pi)^k}.  
\end{align*}
Letting $\tau = 2^{-1/k}$ and $r = \frac{(\Gamma(k+1))^{1/k}}{\tau h_{\min}(P_{\Scal}\Pi)}$ gives \[\EE[\mathrm{diam}(Z_{0, \Scal})^k] \leq \frac{2(1 + c_{r}) \Gamma(k+1)}{h_{\min}(P_{\Scal}\Pi)^k},\] 
and by \eqref{e:zerocell_scaled}, 
\begin{align*}
    \EE[\mathrm{diam}(Z_0^{\lambda})^k] = \frac{1}{\lambda^k} \EE[\mathrm{diam}(Z_{0, \Scal})^k]\leq \frac{c(k,\Pi)}{\lambda^kh_{\min}(P_{\Scal}\Pi)^k},
\end{align*}
where $c(k,\Pi) := 2(1 + c_{r}) \Gamma(k+1)$ for the chosen $r= r(k, \Pi)$.
\end{proof}

\begin{remark}
For the isotropic model as in Example \ref{ex:isotropic}, the normalized associated zonoid is the ball $c_DB^D$, and $h_{\min}(\Pi) = c_D$. In fact, $h_{\min}(\Pi)$ is maximal in the isotropic case because for any other $\phi$ and $\Pi$ defined by \eqref{eq:hZ}, there will be a direction $v$ for which $h(\Pi, v) \leq \int_{\mathbb{S}^{D-1}} h(\Pi, u) \dint \sigma_{D-1}(u) = c_D$, where $\sigma_{D-1}$ is the uniform distribution on $\mathbb{S}^{D-1}$. Thus, the exponential rate of the tail bound \eqref{e:diamtail} is maximal for the isotropic STIT tessellation.
\end{remark}

\subsection{Number of cells in a compact domain}

The following upper bound on the number of cells of $\mathcal{P}(\lambda)$ that intersect a compact and convex subset of $\RR^D$ follows from equation \eqref{e:campbell}.

\begin{lemma}\label{l:number}
Let $K$ be a compact and convex set contained in a $d$-dimensional subspace $\mathcal{S}$ of $\RR^D$.  
Let $N_{\lambda}(K)$ be the number of cells of $\mathcal{P}(\lambda)$ that intersect $K$. Then,
\begin{align*}
   \EE[N_{\lambda}(K)] & = \mathrm{vol}_d(P_{\mathcal{S}}\Pi)\sum_{k=0}^{d}  \binom{d}{k} \lambda^k \EE[V(K[k], Z_{\mathcal{S}}[d-k])],
\end{align*}
where $Z_{\mathcal{S}}$ is the typical cell of the STIT tessellation in $\Scal$ with associated zonoid $P_{\Scal}\Pi$ and $\EE[V(K[k], Z[r-k])]:= \EE[V(\underbrace{K,\ldots,K}_{k}, \underbrace{Z, \ldots, Z}_{r-k})]$.
\end{lemma}

The \emph{mixed volume} $V(K_1, \ldots, K_d)$ of a collection of convex bodies $K_1, \ldots, K_d$ is non-negative, translation-invariant, multilinear, and symmetric in its arguments \cite[Section 5.1]{SchneiderBook}.
For $k \in \mathbb{N}$, let $B^k$ denote the unit ball in $\RR^k$, and define $\kappa_{k} := \mathrm{vol}_k(B^k)$. The \emph{intrinsic volumes} of a convex body $K \subset \RR^d$ are defined for $j=0, \ldots, d$ by
\[V_j(K) := \frac{\binom{d}{j}}{\kappa_{d-j}}V(K[j], B^d[d-j]).\]
The case $j=d$ is the usual volume, i.e. $V_d = \mathrm{vol}_d$. If $K$ has dimension $j < d$, the normalization ensures $V_j(K)$ is the usual $j$-dimensional volume $\mathrm{vol}_j(K)$.

\begin{example}\label{ex:Nlambda_ex1}
If $K = RB^d$, a ball of radius $R$ in $\RR^d$, then
\begin{align*}
\EE[N_{\lambda}(RB^d)] &= \mathrm{vol}_d(\Pi)\sum_{k=0}^d  \binom{d}{k} \lambda^k R^k \EE[V(B^d[k], Z[d-k])] \\  
&= \mathrm{vol}_d(\Pi)\sum_{k=0}^d \lambda^k R^k \kappa_k  \EE[V_{d-k}(Z)].  
\end{align*}
By (10.3) and Theorem 10.3.3 in \cite{weil}, $\EE V_{d-k}(Z) =  \frac{V_{k}(\Pi)}{\mathrm{vol}_d(\Pi)}$. Thus,
\begin{align}\label{e:Nlambda_Ball}
   \EE[N_{\lambda}(RB^d)] = \sum_{k=0}^d (\lambda R)^k \kappa_k  V_k(\Pi).
\end{align}
\end{example}
\begin{example}\label{ex:Nlambda_ex2}
For the isotropic STIT (see Example \ref{ex:isotropic}) and any convex body $K \subset \RR^D$, Proposition 3 in \cite{Thale2011} gives
\begin{align*}
    \EE[N_{\lambda}(K)] = \sum_{k=0}^D \left(\prod_{j=1}^k \gamma_j\right)\frac{\lambda^k}{k!}V_k(K),
\end{align*}
where the constant $\gamma_j$ is defined as $\gamma_j := \frac{\Gamma(\frac{j+1}{2})\Gamma(\frac{D}{2})}{\Gamma(\frac{j}{2})\Gamma(\frac{D+1}{2})}$.
\end{example}
\begin{example}\label{ex:Nlambda_ex3}
Suppose $K = [0,1]^D$ and $Z$ is the typical cell of a STIT tessellation with directional distribution $\phi = \frac{1}{2D}\sum_{i=1}^D(\delta_{e_i} + \delta_{-e_i})$ and lifetime parameter $D$. This setting corresponds with the Mondrian process in \cite{mourtada2020minimax}. 
Then, $h(Z, u) = \frac{T_i}{2}\sum_{i=1}^D |\langle u, e_i\rangle|$, where $T_1, \ldots T_d$ are i.i.d. exponential random variables with unit mean. By the formula for mixed volumes of zonoids from \cite[p.~614]{weil},
\[V(K[k], Z[d-k]) \stackrel{d}{=} \prod_{i=1}^{D} T_i,\]
and $\EE[V(K[k], Z[d-k])] = 1$.
Lemma \ref{l:number} then implies
\[\EE[N_{\lambda}([0,1]^D)] = \sum_{k=0}^D  \binom{d}{k} \lambda^k = (1 + \lambda)^D,\]
which recovers Proposition 2 in \cite{mourtada2020minimax}.
\end{example}

\begin{proof}(of Lemma \ref{l:number})
Let $Z_{\lambda}$ denote the typical cell of $\mathcal{P}(\lambda)$.
We first note that by Fact \ref{f:intersection} and \eqref{e:zerocell_scaled}, the random polytope $Z_{\lambda} \cap \mathcal{S}$ has the same distribution as $\lambda^{-1}Z_{\mathcal{S}}$, where $Z_{\Scal}$ is as defined in the Lemma.
Then, applying \eqref{e:campbell} with the indicator function $f(\cdot ) = 1_{\{\cdot \cap K \neq \emptyset\}}$ and \eqref{e:EVZ} gives that the expected number of cells of $\mathcal{P}(\lambda)$ intersecting $K \subset \mathcal{S}$ satisfies
\begin{align*}
   \EE[N_{\lambda}(K)] &= \EE\left[\sum_{C \in \mathcal{P}(\lambda)} 1_{\{C \cap K \neq \emptyset\}}\right] = \EE\left[\sum_{C \in \mathcal{P}(\lambda) \cap \mathcal{S}} 1_{\{C \cap K \neq \emptyset\}}\right]\\
   &= \frac{\lambda^d}{\EE[\mathrm{vol}_d(Z_{\Scal})]}\EE\left[\int_{\Scal} 1_{\{\lambda^{-1}Z_{\Scal} + y \cap K \neq \emptyset\}} \dint y \right] \\
   &= \lambda^d \mathrm{vol}_d(P_{\Scal}\Pi)\EE[\mathrm{vol}_d(K - \lambda^{-1}Z_{\Scal})] \\
   &=  \lambda^d\mathrm{vol}_d(P_{\Scal}\Pi) \sum_{k=0}^d \binom{d}{k}\EE[V(K[k], -\lambda^{-1}Z_{\Scal}[d-k])].
\end{align*}
The last equality follows from the relation (5.16) in \cite{weil}. The third equality follows from the fact that $\lambda^{-1}Z_{\Scal} + y \cap K \neq \emptyset$ if and only if $y \in K - \lambda^{-1}Z_{\Scal}$. 
By the scaling property of mixed volumes and the fact that $Z \stackrel{d}{=} -Z$,
\[\EE[V(K[k], -\lambda^{-1}Z_{\Scal}[d-k])] = \lambda^{-(d-k)}\EE[V(K[k], Z_{\Scal}[d-k])].\]
Thus,
\begin{align*}
  \EE[N_{\lambda}(K)] &=  \mathrm{vol}_d(P_{\Scal}\Pi)\sum_{k=0}^d  \binom{d}{k} \lambda^k \EE[V(K[k], Z_{\Scal}[d-k])].%,
\end{align*}
%which proves the claim.
\end{proof}

%%%%%%%%%%%%%%%%%%%%%%%%%%%%%%%%%%%%%%%%%%%%%%%%%%%%%%%%%%%%%%%

\section{Main results}\label{s:main}

%\subsection{Statements of the main results} 
Fix a non-empty compact and convex $D$-dimensional domain $W \subset \RR^D$, and consider the following regression setting. The data set $\mathcal{D}_n := \{(X_1, Y_1), \ldots, (X_n,Y_n)\}$ consists of $n$ i.i.d. samples from a random pair $(X,Y) \in W \times \RR$ such that $\EE[Y^2] < \infty$.  Let $\mu$ denote the unknown distribution of $X$ and 
\[Y = f(X) + \ee,\]
where $f(X) = \EE[Y|X]$ is the conditional expectation of $Y$ given $X$ and $\ee$ is noise such that $\EE[\ee|X] = 0$ and $\mathrm{Var}(\ee|X) = \sigma^2 < \infty$ almost surely.

Let $\mathcal{P}$ be a random tessellation of $W$. The regression tree estimator based on $\mathcal{P}$ is
\begin{align}\label{e:tree}
    \hat{f}_n(x, \mathcal{P}) := \sum_{i=1}^n \frac{1_{\{X_i \in Z_x\}}}{\mathcal{N}_n(x)}Y_i,
\end{align}
where $Z_x$ is the cell of $\mathcal{P}$ that contains $x$ and $\mathcal{N}_n(x) := \sum_{i=1}^n 1_{\{X_i \in Z_x\}}$ is the number of points in $Z_x$. If $\mathcal{N}_n(x) = 0$, then it is assumed that $\hat{f}_n(x, \mathcal{P}) = 0$. The random forest estimator based on $\mathcal{P}$ is defined by averaging $M$ i.i.d. copies of the tree estimator, i.e.
\begin{align}\label{e:forest}
    \hat{f}_{n,M}(x) := \frac{1}{M} \sum_{m=1}^M \hat{f}_n(x, \mathcal{P}_m),
\end{align}
where $\mathcal{P}_1, \ldots, \mathcal{P}_M$ are $M$ i.i.d. copies of $\mathcal{P}$. 

We define the STIT regression tree estimator $\hat{f}_{\lambda,n}$ and the STIT regression forest estimator $\hat{f}_{\lambda,n, M}$ as in \eqref{e:tree} and \eqref{e:forest} respectively, where $\mathcal{P} := \mathcal{P}(\lambda) \cap W$ is the random tessellation on $W$ generated by a STIT tessellation $\mathcal{P}(\lambda)$ with lifetime parameter $\lambda$ and normalized associated zonoid $\Pi$.
The quality of the estimator $\hat{f}_{\lambda,n,M}$ is measured by the quadratic risk
\[R(\hat{f}_{\lambda,n,M}) := \EE[(\hat{f}_{\lambda,n, M}(X) - f(X))^2].\]

We now define the function classes we will consider in our results. For $k \in \mathbb{N}$, $\beta \in (0,1]$, and $L > 0$, define the $(k + \beta)$-H\"{o}lder ball of norm $L$, denoted by $\mathcal{C}^{k, \beta}(L) = \mathcal{C}^{k, \beta}(W, L)$, to be the set of all $k$ times differentiable functions $f: W \to \RR$ such that for all multi-indices $\alpha$ with $|\alpha| \leq k$,
\[\|D^{\alpha}f(x) - D^{\alpha}f(y)\| \leq L\|x - y\|^{\beta} \text{  and  } \|D^{\alpha}f(x)\| \leq L, \]
for all $x,y \in W$. %We assume here that $W$ is a compact and convex $d$-dimensional domain $W \subset \RR^d$. 
The minimax rate for the class $\mathcal{C}^{k, \beta}(L)$ is $n^{-2(k + \beta)/(2(k + \beta) + D)}$ \cite[Theorem 3.2]{Gyorfi}. 

Our main results show that for an appropriate choice of $\lambda$, STIT forest estimators achieve the minimax rate of convergence for $\mathcal{C}^{0, \beta}(L)$ and $\mathcal{C}^{1, \beta}(L)$. The rates are also adaptive to the intrinsic dimension of the input data, defined as follows.

\begin{definition}
The input $X$ has intrinsic dimension $d$ if the support of its distribution $\mu$ is contained in an $d$-dimensional linear subspace $\mathcal{S}$ of $\RR^D$.
\end{definition}

We now state our first main result.

\begin{theorem}\label{t:rate1}
 Assume $X$ has intrinsic dimension $d$ and $f \in \mathcal{C}^{0, \beta}(L)$ for $\beta \in (0,1]$ and $L > 0$.
Then, 
\begin{align}\label{e:risk_bound}
R(\hat{f}_{\lambda,n,M}) &\leq  \frac{L^2c_{\beta, P_{\Scal}\Pi}}{\lambda^{2\beta}h_{\min}(P_{\Scal}\Pi)^{2\beta}}  + \frac{(5\|f\|^2_{\infty} + 2\sigma^2)\mathrm{vol}_d(P_{\mathcal{S}}\Pi)}{n}\sum_{k=0}^{d}  \binom{d}{k} \lambda^k \EE\left[V\left(W_{\Scal}[k], Z_{\Scal}[d-k]\right)\right], 
\end{align}
where $W_{\Scal} := W \cap \Scal$ and $Z_{\Scal}$ is the typical cell of the STIT tessellation in $\Scal$ with associated zonoid $P_{\Scal}\Pi$.
If $\mathrm{diam}(W) \leq 2R$, then
\begin{align}\label{e:risk_bound_B}
R(\hat{f}_{\lambda,n,M})&\leq \frac{L^2c_{\beta, P_{\Scal}\Pi}}{\lambda^{2\beta}h_{\min}(P_{\Scal}\Pi)^{2\beta}} + \frac{(5\|f\|^2_{\infty} + 2\sigma^2)}{n}\sum_{k=0}^{d}  \lambda^k R^k \kappa_k V_k\left(P_{\Scal}\Pi\right). 
\end{align}
\end{theorem}

\begin{corollary}\label{cor:rate1}
In the setting of Theorem \ref{t:rate1}, letting $\lambda_n \sim L^{2/(d + 2 \beta)}n^{1/(d + 2\beta)}$ as $n \to \infty$ yields
\begin{align}\label{e:rate1}
R(\hat{f}_{\lambda_n,n,M})= O\left( L^{2d/(d + 2\beta)}n^{-2\beta/(d+2\beta)}\right),
\end{align}
which is the minimax rate for the class $\mathcal{C}^{0, \beta}(L)$ on $\RR^d$.
\end{corollary}

The minimax rate above holds even for STIT tree estimators. To see the advantage of averaging multiple trees in a random forest estimator, the following result assumes additional smoothness on the function $f$. In this case, the minimax rate is only achieved when the forest size $M$ is large enough.

\begin{theorem}\label{t:rate2}
Assume $X$ has intrinsic dimension $d$ and the distribution $\mu$ of $X$ has a positive and Lipschitz density with respect to the Lebesgue measure on its $d$-dimensional convex support $K$. Let $r(K)$ denote the inradius of $K$ and define $K_{\ee} := \{x \in K : d(x, \partial K) \geq \ee\}$. Assume $f \in \mathcal{C}^{1, \beta}(L)$ for $\beta \in (0,1]$ and $L > 0$.  Then, for $\ee \in (0, r(K))$,
\begin{align*}
\EE[(\hat{f}_{\lambda,n, M}(X) - f(X))^2 | X \in K_{\ee}] \leq O\left( \frac{L^2}{\lambda^2M} + \frac{L^2}{\lambda^{2\beta + 2}} + \frac{\lambda^d}{n}\right).
\end{align*}
In the unconditional case when $\ee = 0$,
\begin{align*}
\EE[(\hat{f}_{\lambda,n, M}(X) - f(X))^2] \leq O\left(\frac{L^2}{\lambda^2 M} + \frac{L^2}{\lambda^{\min\{3, 2\beta + 2)}\}} + \frac{\lambda^d}{n}\right).
\end{align*}.
\end{theorem}

\begin{corollary}\label{cor:rate2}
In the setting of Theorem \ref{t:rate2}, choosing
\begin{align*}
    \lambda_n \sim L^{2/(d + 2\beta + 2)}n^{1/(d + 2\beta + 2)} \quad \text{ and } \quad M_n \gtrsim L^{4\beta/(d + 2\beta + 2)}n^{2\beta/(d + 2\beta + 2)}
\end{align*}
as $n \to \infty$ implies
\begin{align}\label{e:rate2}
\EE[(\hat{f}_{\lambda_n,n, M}(X) - f(X))^2 | X \in K_{\ee}] =  O\left(L^{2d/(d + 2\beta + 2)}n^{-(2\beta + 2)/(d + 2\beta + 2)}\right),
\end{align}
which is the minimax rate for the class $\mathcal{C}^{1, \beta}(L)$ on $\RR^d$.
In the unconditional case when $\ee = 0$, the rate \eqref{e:rate2} is obtained when $\beta \leq 1/2$. When $\beta > 1/2$, choosing
\begin{align*}
    \lambda_n \sim L^{2/(d + 3)}n^{1/(d+3)} \quad \text{ and } M_n \gtrsim L^{2/(d + 3)}n^{1/(d+3)} 
\end{align*}
as $n \to \infty$ implies
\begin{align*}
\EE[(\hat{f}_{\lambda_n,n, M}(X) - f(X))^2]  = O\left(L^{2d/(d+3)}n^{-3/(d+3)}\right).
\end{align*}
\end{corollary}

\begin{remark}\label{Rem:adaptivity}
Specialized to the case of the Mondrian, our rates are an improvement over the results of \cite{mourtada2020minimax}, where no notion of the low dimensionality of the input was considered. However, note that the rates are obtained through optimal choices of $\lambda$ and $M$ that depend on $d$ and $\beta$. In practice, the intrinsic dimension of the input and the regularity of $f$ are not known \emph{a priori}, and it is an open problem to find an adaptive way of choosing the lifetime parameter $\lambda$ such that the random forest estimator achieves these optimal rates without this prior knowledge.
\end{remark}

The upper bounds on the risk in Theorem \ref{t:rate1} and \ref{t:rate2} only rely on statistics of the typical cell and the zero cell of the STIT tessellation $\mathcal{P}(\lambda)$. Since these values are identical for a STIT and a stationary Poisson hyperplane tessellation with matching parameters, it follows that regression estimators based on stationary Poisson hyperplane processes have \emph{identical} risk bounds. We state this formally as follows. 
\begin{theorem}\label{t:poisson}
Let $\hat{f}_{\lambda, n, M}$ be the random forest estimator defined as in 
\eqref{e:forest} using $M$ i.i.d. random tessellations induced by a stationary Poisson hyperplane process with intensity $\lambda$ and normalized associated zonoid $\Pi$. Then, the minimax optimal convergence rates \eqref{e:rate1} and \eqref{e:rate2} hold for the risk of $\hat{f}_{\lambda_n,n, M}$ in each corresponding setting.
\end{theorem}

\begin{remark}
While the rates in Corollaries \ref{cor:rate1} and \ref{cor:rate2} depend only on the intrinsic dimension of $X$, the ambient dimension $D$ appears in the upper bounds of Theorems \ref{t:rate1} and \ref{t:rate2} through the constants. To clarify the dependence on $D$, we insert $\lambda = L^{1/(d+2\beta)}n^{1/(d + 2\beta)}$ into \eqref{e:risk_bound_B} to obtain
\begin{align*}%\label{e:risk_bound_B}
R(\hat{f}_{\lambda,n,M})&\leq \left(\frac{c_{\beta, P_{\Scal}\Pi}}{h_{\min}(P_{\Scal}\Pi)^{2\beta}} + (5\|f\|^2_{\infty} + 2\sigma^2) R^d \kappa_d \mathrm{vol}_d\left(P_{\Scal}\Pi\right) \right)L^{\frac{d}{d+2\beta}}n^{\frac{-2\beta}{d + 2\beta}} + o\left(n^{\frac{-2\beta}{d + 2\beta}}\right).
\end{align*}
Since the leading order constant involves geometric properties of $P_{\Scal}\Pi$, the order with respect to $D$ will depend on the subspace $\mathcal{S}$ and the shape of $\Pi$. We first note that the constant $\mathrm{vol}_d(P_{\mathcal{S}}\Pi)$ has a uniform upper bound that does not depend on $D$. Indeed, since $h(\Pi,u) \leq 1$ for all $u \in \mathbb{S}^{D-1}$, %we have that $\Pi \subseteq B^D$, and $P_{\Scal}\Pi \subseteq B^D \cap \Scal$. Thus,
$\mathrm{vol}_d(P_{\mathcal{S}}\Pi) \leq \kappa_d$ for any subspace $\Scal$ of dimension $d$. 
However, the constant $h_{\min}(P_{\Scal}\Pi)^{-2\beta}$ could grow faster with $D$ if the subspace $\mathcal{S}$ lies in ``bad" directions with respect to $\Pi$. This observation highlights a potential downside of using axis-aligned splits. Indeed, for the Mondrian (see Example \ref{ex:mondrian}), the associated zonoid is $\Pi = \frac{1}{D}[-1,1]^D$ and if $\Scal = \text{span}(\mathbf{1})$, then $h_{\min}(P_{\Scal}\Pi)^{-2\beta} = O(D^{\beta})$. However, if $\mathcal{S} = \mathrm{span}(e_1)$, then $h_{\min}(P_{\Scal}\Pi)^{-2\beta} = O(D^{2\beta})$. Without further knowledge of the subspace one can instead choose an isotropic STIT (see Example \ref{ex:isotropic}), where the associated zonoid is $\Pi = c_DB^D$ and $h_{\min}(P_{\Scal}\Pi)^{-2\beta} = O(D^{\beta})$ for any subspace $\Scal$. Unfortunately due to the unknown dependence of the constant $c_{\beta, P_{\Scal}\Pi}$ on $D$ we cannot currently make this analysis more precise.
We leave for future work a more thorough study of the statistical advantages of oblique splits, see Section \ref{s:discussion}. %Indeed, we hope that for certain sets of structural assumptions about the data source and/or underlying function, one can obtain such a characterization via the geometric properties of the associated zonoid.
\end{remark}

\subsection{Rates for binary classification}\label{sec:classification}

We next show that we can extend all of the above results to binary classification as was done for Mondrian random forests in Section 5.5 of \cite{mourtada2020minimax}. In this setting, we assume the data are i.i.d. samples from a random pair $(X,Y) \in W \times \{0,1\}$. We then define the function $\eta(x) := \PP(Y = 1 | X = x)$ and the optimal classifier $g(x) := 1_{\{\eta(x) \geq 1/2\}}$. The STIT (or Poisson hyperplane) forest classifier $\hat{g}_{\lambda, n, M}$ is defined by
\[\hat{g}_{\lambda, n, M}(x) := 1_{\{\hat{\eta}_{\lambda, n, M}(x) \geq 1/2\}}, \qquad x \in W,\]
where $\hat{\eta}_{\lambda, n, M}(x)$ is the STIT (or Poisson hyperplane) forest estimator of $\eta$ as defined in the previous section. The risk of $\hat{g}_{\lambda, n , M}$ is given by the classification error
\[\mathcal{L}(\hat{g}_{\lambda, n , M}) := \PP(\hat{g}_{\lambda, n , M}(X) \neq Y).\]
To evaluate the classification estimator, we compare $\mathcal{L}(\hat{g}_{\lambda, n , M})$ to the Bayes risk
$\mathcal{L}(g) := \PP(g(X) \neq Y)$ and consider the difference
\[R(\hat{g}_{\lambda, n , M}) := \mathcal{L}(\hat{g}_{\lambda, n , M}) - \mathcal{L}(g).\]
A general theorem \cite[Theorem 6.5]{LugosiBook} shows that $R(\hat{g}_{\lambda, n , M})$ is controlled by the square root of the risk of the regression estimator $\hat{\eta}_{\lambda, n, M}$ and implies the following Corollary of Theorem \ref{t:rate1}.
\begin{corollary}\label{cor:class1}
Assume $X$ has intrinsic dimension $d$ and $\eta \in \mathcal{C}^{0, 1}(L)$ for $L > 0$.
Then, letting $\lambda_n \sim n^{1/(d + 2\beta)}$ as $n \to \infty$ yields
\begin{align}
R(\hat{g}_{\lambda_n,n,M})= o\left(n^{-\beta/(d+2\beta)}\right).
\end{align}
\end{corollary}
%This rate correspond to the minimax rate of convergence for these classes \cite{Yang1999}.
We similarly obtain the following Corollary of Theorem \ref{t:rate2} showing that we can obtain a faster rate with forest estimators than with single trees.
\begin{corollary}\label{cor:class2}
Assume $X$ has intrinsic dimension $d$ and the distribution $\mu$ of $X$ has a positive and Lipschitz density with respect to Lebesgue measure on its $d$-dimensional convex support $K$. Let $r(K)$ denote the inradius of $K$ and define $K_{\ee} := \{x \in K : d(x, \partial K) \geq \ee\}$. Assume $\eta \in \mathcal{C}^{1, \beta}(L)$ for $\beta \in (0,1]$ and $L > 0$.  Then, for $\ee \in (0, r(K))$, choosing
\begin{align*}
    \lambda_n \sim n^{1/(d + 2\beta + 2)} \quad \text{ and } \quad M_n \gtrsim n^{2\beta/(d + 2\beta + 2)}
\end{align*}
as $n \to \infty$ implies
\begin{align}
\PP\left(\hat{g}_{\lambda_n, n, M_n}(X) \neq Y | X \in K_{\ee}\right) - \PP\left(g(X) \neq Y | X \in K_{\ee}\right) =  o\left(n^{-(\beta + 1)/(d + 2\beta + 2)}\right).
\end{align}
\end{corollary}

\section{Proofs}\label{sec:proofs}

The above results follow from the following bias-variance decomposition of the risk of a tree estimator presented by \cite{arlot2014analysis}. A subtle difference between their setting and ours is that they view the partition as a finite partitioning of $[0,1]^d$, and here we consider the partition to be a stationary STIT tessellation on $\RR^d$ which we view through the compact and convex window $W$ that contains the support of $\mu$. First, let $Z_x^{\lambda}$ denote the cell of $\mathcal{P}(\lambda)$ that contains the vector $x \in \RR^d$, and define
\[\bar{f}_{\lambda}(x) := \EE_X[f(X) | X \in Z^{\lambda}_x], \quad x \in W.\]
Conditioned on $\mathcal{P}(\lambda)$, this is the orthogonal projection of $f \in L^2(W, \mu)$ onto the subspace of functions that are constant within the cells of $\mathcal{P}(\lambda) \cap W$.

Conditioned additionally on the data $\mathcal{D}_n$, the random tree estimator $\hat{f}_{\lambda,n}$ is in this subspace of piecewise functions, and hence $\EE_X[(f(X) - \bar{f}_{\lambda}(X))\hat{f}_{\lambda,n}(X)] = 0$. Thus, given $\mathcal{P}(\lambda)$ and $\mathcal{D}_n$,
\begin{align*}
    \EE_X[(f(X) - \hat{f}_{\lambda,n}(X))^2] &= \EE_X[(f(X) - \bar{f}_{\lambda}(X) + \bar{f}_{\lambda}(X) - \hat{f}_{\lambda,n}(X))^2]  \\
   & = \EE_X[(f(X) - \bar{f}_{\lambda}(X))^2] + \EE_X[(\bar{f}_{\lambda}(X) - \hat{f}_{\lambda,n}(X))^2].
\end{align*}
Taking the expectation over $\mathcal{P}(\lambda)$ and $\mathcal{D}_n$ gives the following decomposition of the risk:
\begin{align}\label{e:tree_decomp}
    R(\hat{f}_{\lambda,n}) :=  \EE[(f(X) - \hat{f}_{\lambda,n}(X))^2] = \EE[(f(X) - \bar{f}_{\lambda}(X))^2] + \EE[(\bar{f}_{\lambda}(X) - \hat{f}_{\lambda,n}(X))^2].
\end{align}
The first term measures how far away $f$ is from the closest function in the hypothesis class that the estimators lie in and is called the approximation error or bias. The second term measures the estimation error, or variance, coming from the fact that we build the estimator from only a finite number of samples. As in the results of \cite{mourtada2020minimax}, the bias and variance depend on the geometric properties of the cells of the tessellations from which the estimator is built. In particular, the bias is controlled by moments of the diameter of the zero cell, and the variance is controlled by the expected number of cells that have a non-empty intersection with the support of $\mu$. Lemmas \ref{l:diameter} and \ref{l:number} provide the needed bounds, and choosing an optimal lifetime $\lambda$ depending on the number of samples and Lipschitz constant gives the results.

%\subsection*{Convention about constants} 
%In the following, $c_1,c_2,\ldots$ denote positive constants which depend on the data that are indicated as arguments.

%\end{remark}}

\subsection{Variance Bound}
In the following, we see that the variance term can by controlled by the expected number of cells of the tessellation that intersect the support of $\mu$.%, as in \cite{mourtada2020minimax}.

\begin{lemma}\label{l:variance}
Let $N_{\lambda}(K)$ be the number of cells of $\mathcal{P}(\lambda)$ that have non-empty intersection with a bounded subset $K \subset \RR^D$. Then,
\begin{align*}
    \EE\left[(\bar{f}_{\lambda}(X) - \hat{f}_{\lambda,n}(X))^2\right] \leq \frac{5\|f\|_{\infty}^2 + 2\sigma^2}{n} \EE[N_{\lambda}(\mathrm{supp}(\mu))].
\end{align*}
\end{lemma}

The proof follows the ideas of \cite[Proposition 2]{arlot2014analysis} which relies crucially on Proposition 1 in \cite{arlot2008}. For completeness and clarity, a proof of this lemma appears below.

\begin{proof}
We first condition on $\mathcal{P}(\lambda)$ and compute the variance of the tree estimator corresponding to a fixed tessellation. Note that the assumption $\mathcal{D}_n$ and $\mathcal{P}(\lambda)$ are independent allows us to take these expectations separately. Also, recall that if no points of $\{X_1, \ldots, X_n\}$ fall in $Z_x^{\lambda}$, then $\hat{f}_{\lambda,n}(x) = 0$. For each $C \in \mathcal{P}(\lambda)$, let $\mathcal{N}_n(C) = \sum_{i=1}^n 1_{\{X_i \in C\}}$ be the number of covariates inside $C$ and let $p_{\lambda,C} := \PP_X(X \in C)$. Then,
\begin{align*}
    &\EE_{\mathcal{D}_n,X}\left[(\bar{f}_{\lambda}(X) - \hat{f}_{\lambda,n}(X))^2\right] \\
    &= \int_{\RR^d} \sum_{C \in \mathcal{P}(\lambda)} 1_{\{x \in C\}} \EE_{\mathcal{D}_n} \left[ \left(\EE_X[f(X)| X \in C] - \frac{\sum_{i=1}^n Y_i1_{\{X_i \in C\}}}{\mathcal{N}_n(C)}\right)^2\right] \dint \mu(x) \\
     &= \sum_{C \in \mathcal{P}(\lambda): C \cap \mathrm{supp}(\mu) \neq \emptyset} p_{\lambda, C} \EE_{\mathcal{D}_n}\left[\left(\EE_X[f(X)| X \in C] - \frac{\sum_{i=1}^n Y_i1_{\{X_i \in C\}}}{\mathcal{N}_n(C)}\right)^2\right].
\end{align*}
The expectation in the sum satisfies
\begin{align*}
&\EE_{\mathcal{D}_n}\left[\left(\EE_X[f(X)| X \in C] - \frac{\sum_{i=1}^n Y_i1_{\{X_i \in C\}}}{\mathcal{N}_n(C)}\right)^2\right] \\
    &= \sum_{k=1}^n\PP(\mathcal{N}_n(C) = k) \EE_{\mathcal{D}_n}\left[\left(\EE_X[f(X)| X \in C] - \frac{\sum_{i=1}^n Y_i1_{\{X_i \in C\}}}{k} \right)^2\bigg| \mathcal{N}_n(C) = k\right] \\
    & \qquad + \PP(\mathcal{N}_n(C) = 0) \EE_X[f(X)| X \in C]^2.
\end{align*}
By the assumptions on the noise, the conditional expectation in the sum satisfies 
\begin{align*}
&\EE_{\mathcal{D}_n}\left[\left(\EE_X[f(X)| X \in C] - \frac{\sum_{i=1}^n Y_i1_{\{X_i \in C\}}}{k}\right)^2 \bigg| \, \mathcal{N}_n(C) = k\right] \\
&=  k^{-2}\EE_{\mathcal{D}_n}\left[\left(k\EE_X[f(X)| X \in C] - \sum_{i=1}^n (f(X_i) + \ee_i)1_{\{X_i \in C\}} \right)^2 \bigg|\, \mathcal{N}_n(C) = k\right]\\
&=  k^{-2}\sum_{i_1< \cdots < i_k}\PP(X_{i_1}, \ldots, X_{i_k} \in C | \, \mathcal{N}_n(C) = k)\\
& \cdot \EE_{\mathcal{D}_n}\left[\left( k\EE_X[f(X)| X \in C] - \sum_{j=1}^kf(X_{i_j}) - \sum_{j=1}^k\ee_{i_j}\right)^2 \bigg| \, \mathcal{N}_n(C) = k, X_{i_1}, \ldots, X_{i_k} \in C\right] \\
&=  k^{-2}\EE_{\mathcal{D}_n}\left[\left(k\EE_X[f(X)| X \in C] - \sum_{i=1}^kf(X_{i}) - \sum_{i=1}^k\ee_{i}\right)^2 \bigg|\, X_{1}, \ldots, X_{k} \in C\right] \\
&= k^{-2}\EE_{\mathcal{D}_n}\left[\left(k\EE_X[f(X)|\, X \in C] - \sum_{i=1}^k f(X_{i})\right)^2 \bigg| \, X_1, \ldots, X_k \in C\right] + k^{-1}\sigma^2,
\end{align*}
and by the independence of the $X_i$'s, the expectation in the first term simplifies to
\begin{align*}
    &\EE_{\mathcal{D}_n}\left[\left(k\EE_X[f(X)| X \in C] - \sum_{i=1}^kf(X_{i})\right)^2\bigg| X_1, \ldots, X_k \in C\right] \\
    &= k^2\EE_X[f(X)| X \in C]^2 - 2k^2\EE_X[f(X)| X \in C]^2  \\
    &\quad \quad + \EE_{\mathcal{D}_n}\left[\sum_{i,j=1}^k f(X_i)f(X_j) \bigg| \, X_1, \ldots, X_k \in C\right] \\
    &= k\EE_{X}[f(X)^2 | X \in C] + (k^2 - k)\EE_X[f(X)| X \in C]^2 - k^2\EE_X[f(X)| X \in C]^2 \\
    &= k(\EE_{X}[f(X)^2 | X \in C] - \EE_X[f(X)| X \in C]^2).
\end{align*}
Thus,
\begin{align*}
&\EE_{\mathcal{D}_n}\left[\left(\EE_X[f(X)| X \in C] - \frac{\sum_{i=1}^n Y_i1_{\{X_i \in C\}}}{\mathcal{N}_n(C)}\right)^2\right] \\
&=  \sum_{k=1}^n\PP(\mathcal{N}_n(C) = k) k^{-1}\left(\EE_{X}[f(X)^2 | X \in C] - \EE_X[f(X)| X \in C]^2 + \sigma^2\right) \\
    & \qquad + \PP(\mathcal{N}_n(C) = 0) \EE_X[f(X)| X \in C]^2 \\
    &=\left(\EE_{X}[f(X)^2 | X \in C] - \EE_X[f(X)| X \in C]^2 + \sigma^2\right)\sum_{k=1}^n  \binom{n}{k} p_{\lambda, C}^{k} (1 - p_{\lambda, C})^{n-k} k^{-1}\\
    & \qquad + \EE_X[f(X)| X \in C]^2(1 - p_{\lambda,C})^n \\
    &\leq \left(2\|f\|_{\infty}^2 + \sigma^2 \right) \sum_{k=1}^n  \binom{n}{k} p_{\lambda, C}^{k} (1 - p_{\lambda, C})^{n-k} k^{-1} + \|f\|_{\infty}^2(1 - p_{\lambda,C})^n.
\end{align*}
Now, note that for $B \sim \mathrm{Binomial}(n, p_{\lambda, C})$,
\begin{align*}
 \sum_{k=1}^n \binom{n}{k} np_{\lambda, C}^{k+1} (1 - p_{\lambda, C})^{n-k} k^{-1} &= \EE[B] \EE[B^{-1}1_{\{B > 0\}}],
\end{align*}
and $\EE[B] \EE[B^{-1}1_{\{B > 0\}}] \leq \frac{2n p_{\lambda,C}}{(n+1)p_{\lambda,C}} \leq 2$ \cite[Lemma 4.1]{Gyorfi}. Also, the upper bounds $1 - x \leq e^{-x}$ and $xe^{-x} \leq e^{-1}$ for all $x \geq 0$ imply
\begin{align*}
  np_{\lambda, C}(1 - p_{\lambda,C})^n \leq  e^{-1} \leq 1. 
\end{align*}
Thus,
\begin{align*}
    \EE_{\mathcal{D}_n,X}\left[(\bar{f}_{\lambda}(X) - \hat{f}_{\lambda,n}(X))^2\right] &\leq \frac{1}{n} \sum_{\substack{C \in \mathcal{P}(\lambda): \\ C \cap \mathrm{supp}(\mu) \neq \emptyset}}  \left(2\|f\|_{\infty}^2 + \sigma^2 \right) \sum_{k=1}^n  \binom{n}{k} np_{\lambda, C}^{k+1} (1 - p_{\lambda, C})^{n-k} k^{-1} \\
    & \qquad \qquad + \frac{\|f\|_{\infty}^2}{n} \sum_{\substack{C \in \mathcal{P}(\lambda):\\ C \cap \mathrm{supp}(\mu) \neq \emptyset}} np_{\lambda, C}(1 - p_{\lambda,C})^n \\
    &\leq \frac{5\|f\|_{\infty}^2 + 2\sigma^2}{n} N_{\lambda}(\mathrm{supp}(\mu)) .
\end{align*}
Taking the expectation with respect to $\mathcal{P}(\lambda)$ completes the proof.
\end{proof}

\subsection{Proof of Theorem \ref{t:rate1}}

Following the proof of Theorem 2 by \cite{mourtada2020minimax}, we first use Jensen's inequality to reduce the risk of $\hat{f}_{\lambda, n, M}$ to that of a single Mondrian tree estimator $\hat{f}_{\lambda,n} := \hat{f}_{\lambda, n, 1}$. Using the bias-variance decomposition \eqref{e:tree_decomp}, the risk of $\hat{f}_{\lambda,n}$ is given by
\begin{align}
    R(\hat{f}_{\lambda,n}) =  \EE[(f(X) - \hat{f}_{\lambda,n}(X))^2] = \EE[(f(X) - \bar{f}_{\lambda}(X))^2] + \EE[(\bar{f}(X) - \hat{f}_{\lambda,n}(X))^2].
\end{align}

We first consider the bias term. For $x \in \mathrm{supp}(\mu)$, by the assumption on $f$,
\begin{align*}
    |f(x) - \bar{f}_{\lambda}(x)| &\leq \frac{1}{\mu(Z^{\lambda}_x)} \int_{Z^{\lambda}_x} \left|f(x) - f(z)\right| \mu(\dint z) \\
    &\leq \frac{1}{\mu(Z^{\lambda}_x)} \int_{Z^{\lambda}_x} L\|x- z\|^{\beta} \mu(\dint z)  \leq L \mathrm{diam}(Z^{\lambda}_x \cap \mathrm{supp}(\mu))^{\beta}.
\end{align*}
Recall that the assumption $X$ has intrinsic dimension $d$ means there exists a $d$-dimensional linear subspace $\mathcal{S}\subseteq \RR^d$ such that $\mathrm{supp}(\mu) \subset \mathcal{S}$. Then by Lemma \ref{l:diameter},
\begin{align}\label{e:bias_bnd}
    \EE[(f(X) - \bar{f}_{\lambda}(X))^2] \leq L^2\EE[\mathrm{diam}(Z_x^{\lambda} \cap \Scal)^{2\beta}] \leq \frac{L^2 c_{\beta,P_S\Pi}}{\lambda^{2\beta}h_{\min}(P_{\Scal}\Pi)^{2\beta}}.
\end{align}
 For the variance bound, Lemma \ref{l:variance} implies
\begin{align*}
 \EE[(\bar{f}_{\lambda}(X) - \hat{f}_{\lambda,n}(X))^2] &\leq \frac{5\|f\|_{\infty}^2 + 2\sigma^2}{n} \EE[N_{\lambda}\left(\mathrm{supp}(\mu)\right)] \leq \frac{5\|f\|_{\infty}^2 + 2\sigma^2}{n} \EE[N_{\lambda}\left(W \cap \Scal\right)] .
\end{align*}
Let $W_{\mathcal{S}} := W \cap \mathcal{S}$. 
Finally by Lemma \ref{l:number}, % and the fact that for $k > d$, $\EE[V(W[k], Z[d-k])] = 0$, %we have for each $i=1, \ldots K$,
\begin{align*}
 \EE[N_{\lambda}\left(W_{\mathcal{S}}\right)] &=  \mathrm{vol}_d(P_{\Scal}\Pi)\sum_{k=0}^d \binom{d}{k} \lambda^k  \EE[V(W_{\Scal}[k], Z_{\Scal}[d-k])]. 
\end{align*}
Then,
\begin{align}\label{e:var_bnd}
    \EE[(\bar{f}_{\lambda}(X) - \hat{f}_{\lambda,n}(X))^2] &\leq \frac{(5\|f\|^2_{\infty} + 2\sigma^2)\mathrm{vol}_d(P_{\Scal}\Pi)}{n}  \sum_{k=0}^d  \binom{d}{k} \lambda^k \EE[V(W_{\Scal}[k], Z_{\Scal}[d-k])].
\end{align}
Combining equations \eqref{e:bias_bnd}  and \eqref{e:var_bnd} gives the first claim. To prove \eqref{e:risk_bound_B}, we note that if $\mathrm{diam}(W) \leq 2R$, then $\EE[N_{\lambda}\left(\mathrm{supp}(\mu)\right)] \leq \EE[N_{\lambda}\left(RB^D \cap \mathcal{S}\right)]$, and the bound follows from the same argument used in Example \ref{ex:Nlambda_ex1} to obtain \eqref{e:Nlambda_Ball}.

Finally, letting $\lambda = \lambda_n \sim L^{2/(d+2\beta)}n^{1/(d+2\beta)}$ as $n \to \infty$ proves Corollary \ref{cor:rate1}. 

\subsection{Proof of Theorem \ref{t:rate2}}
We first need the following technical lemma.
\begin{lemma}\label{l:int_zero}
Let $Z_x^{\lambda}$ be the cell of $\mathcal{P}(\lambda)$ containing the point $x \in \RR^D$. Assume $x \in \mathcal{S} \subseteq \RR^D$ for a linear subspace $\mathcal{S}$ of dimension $d$. Then,
 \[\int_{\mathcal{S}} (z-x)\EE\left[\frac{1_{\{z \in Z_x^{\lambda}\}}}{\mathrm{vol}_d(Z_x^{\lambda} \cap \mathcal{S})}\right] \dint z = 0\]
\end{lemma}

\begin{proof}
%First, recall from Fact \ref{f:intersection} that $\mathcal{P}(\lambda) \cap \mathcal{S}$ is a STIT tessellation that is stationary with respect to translations in $\mathcal{S}$. 
By the stationarity of $\mathcal{P}(\lambda)$ and a change of variable, for $x \in \mathcal{S}$,
\begin{align*}
    \int_{\mathcal{S}} (z-x)\EE\left[\frac{1_{\{z \in Z_x^{\lambda}\}}}{\mathrm{vol}_d(Z_x^{\lambda} \cap \mathcal{S})}\right] \dint z &= \int_{\mathcal{S}} y \EE\left[\frac{1_{\{y \in Z_0^{\lambda}\}}}{\mathrm{vol}_d(Z_0^{\lambda} \cap \mathcal{S})}\right] \dint y.
\end{align*}
Then, for $y \in \Scal$, Fact \ref{f:intersection} and \eqref{e:campbell} imply
\begin{align*}
    \EE\left[\frac{1_{\{y \in Z_0^{\lambda}\}}}{\mathrm{vol}_d(Z_0^{\lambda} \cap \mathcal{S})}\right] = \EE\left[\frac{1_{\{y \in Z_0^{\lambda} \cap \Scal\}}}{\mathrm{vol}_d(Z_0^{\lambda} \cap \mathcal{S})}\right] &= \frac{1}{\EE[\mathrm{vol}_d(Z^{\lambda}_{\Scal})]}\EE\left[\frac{\mathrm{vol}_d(Z^{\lambda}_{\Scal} \cap Z^{\lambda}_{\Scal} - y)}{\mathrm{vol}_d(Z^{\lambda}_{\Scal})}\right],
\end{align*}
where $Z_{\Scal}^{\lambda}$ is the typical cell of $\mathcal{P}(\lambda) \cap \Scal$.
By the fact that volume is translation invariant,
\begin{align*}
 \EE\left[\frac{\mathrm{vol}_d(Z^{\lambda}_{\Scal} \cap Z^{\lambda}_{\Scal} + y)}{\mathrm{vol}_d(Z^{\lambda}_{\Scal} \cap \mathcal{S})}\right]  = \EE\left[\frac{\mathrm{vol}_d(Z^{\lambda}_{\Scal} - y \cap Z^{\lambda}_{\Scal})}{\mathrm{vol}_d(Z_{\Scal}^{\lambda})}\right].
\end{align*}
Thus the integrand $y \EE\left[\frac{1_{\{y \in Z_0^{\lambda}\}}}{\mathrm{vol}_d(Z_0^{\lambda} \cap \mathcal{S})}\right]$ is an odd function, and the integral is zero.
\end{proof}

\begin{proof}(of Theorem \ref{t:rate2})
Similarly to the proof of Theorem 3 by \cite{mourtada2020minimax}, define for each $m$ and $x \in \Scal$, 
\[\bar{f}^{(m)}_{\lambda}(x) := \EE_X[f(X)|X \in Z^{\lambda,(m)}_x],\]
and let $\bar{f}_{\lambda,M}(x) = \frac{1}{M}\sum_{m=1}^M \bar{f}^{(m)}_{\lambda}(x)$. Also define 
\[\tilde{f}_{\lambda}(x) := \EE[\bar{f}^{(m)}_{\lambda}(x)] = \EE\left[\frac{1}{\mu(Z_x^{\lambda})}\int_{Z_x^{\lambda}}f(z) \mu(dz)\right]= \int_{\Scal}f(z) \EE\left[\frac{1_{\{z \in Z_x^{\lambda}\}}}{\mu(Z_x^{\lambda})}\right] \mu(dz),\]
where we have used the fact that the support of $\mu$ is contained in a linear subspace $\Scal$.
The bias-variance decomposition for the risk of a tree estimator can be extended to the random forest estimator as follows \cite[Equation (1)]{arlot2014analysis}:
\begin{align}\label{e:bias-var_forests}
     \EE[(\hat{f}_{\lambda,n, M}(X) - f(X))^2] =  \EE[(f(X)- \bar{f}_{\lambda,M}(X))^2] + \EE[(\bar{f}_{\lambda,M}(X) - \hat{f}_{\lambda,n,M}(X) )^2].
\end{align}
This is due to the fact that $\EE[\hat{f}_{\lambda,n, 1}(x)| \mathcal{P}(\lambda)] = \bar{f}_{\lambda,1}(x)$.
Indeed, by the independence of the $X_i$'s,
\begin{align*}
    \EE_{\mathcal{D}_n}[\hat{f}_{\lambda,n, 1}(x)] &= \frac{1}{n}\EE_{\mathcal{D}_n}\left[\frac{\sum_{i=1}^n Y_i1_{\{X_i \in Z_x\}}}{\mathcal{N}_n(Z_x)}\right] \\
    &= \frac{1}{n}\sum_{k=1}^n\binom{n}{k} \PP_{\mathcal{D}_n}(X_1, \ldots, X_k \in Z_x|\mathcal{N}_n(Z_x) = k) \\
    & \qquad \cdot \EE_{\mathcal{D}_n}\left[\frac{\sum_{i=1}^k f(X_i)1_{\{X_i \in Z_x\}}}{k} \bigg|X_1, \ldots, X_k \in Z_x, \mathcal{N}_n(Z_x) = k \right] \\
    &= \EE_{X}\left[f(X) | X \in Z_x \right] =  \bar{f}_{\lambda, n, 1}(x).
\end{align*}
For the bias term in \eqref{e:bias-var_forests}, Proposition 1 of \cite{arlot2014analysis} implies
\begin{align*}
  \EE[(f(x)- \bar{f}_{\lambda,M}(x) )^2] = \EE[(f(x) - \tilde{f}_{\lambda}(x))^2] + \frac{\mathrm{Var}(\bar{f}^{(1)}_{\lambda}(x))}{M}.  
\end{align*}
We then have the following upper bound on the variance of $\bar{f}^{(1)}_{\lambda}$: for $x \in \Scal$,
\begin{align*}
 \mathrm{Var}(\bar{f}^{(1)}_{\lambda}(x)) \leq \EE\left[(\bar{f}^{(1)}_{\lambda}(x) - f(x))^2\right] \leq L^2\EE[\mathrm{diam}(Z_x^{\lambda} \cap \Scal)^2] \leq \frac{L^2c_{2,P_{\Scal}\Pi}}{\lambda^2},
\end{align*}
where the last inequality follows from Lemma \ref{l:diameter} and stationarity. 
For the variance term in \eqref{e:bias-var_forests}, Jensen's inequality implies
\begin{align*}
  \EE[(\bar{f}_{\lambda,M}(x) - \hat{f}_{\lambda,n,M}(x))^2] \leq \EE[(\bar{f}_{\lambda}^{(1)}(x) - \hat{f}_{\lambda,n,1}(x))^2].
\end{align*}
Thus, taking the expectation with respect to $X$,
\begin{align*}
    \EE[(\hat{f}_{\lambda,n, M}(X) - f(X))^2] \leq \frac{L^2c_{2,P_{\Scal}\Pi}}{M\lambda^2} + \EE[(f(X) - \tilde{f}_{\lambda}(X))^2] +  \EE[(\bar{f}_{\lambda}^{(1)}(X) - \hat{f}_{\lambda,n,1}(X) )^2].
\end{align*}
This upper bound also holds when conditioning on $X \in K_{\varepsilon}$:
\begin{align*}
&\EE[(\hat{f}_{\lambda,n, M}(X) - f(X))^2| X \in K_{\ee}] \leq \frac{L^2c_{2,P_{\Scal}\Pi}}{M\lambda^2}   \\
& \qquad + \EE[(f(X) - \tilde{f}_{\lambda}(X))^2| X \in K_{\ee}] + \EE[(\bar{f}_{\lambda}^{(1)}(X) - \hat{f}_{\lambda,n,1}(X) )^2 | X \in K_{\ee}].
\end{align*}
We can use Lemmas \ref{l:variance} and \ref{l:number} to bound the variance as before:
\begin{align*}
&\EE[(\bar{f}_{\lambda}^{(1)}(X) - \hat{f}_{\lambda,n,1}(X) )^2] \\
&\qquad \leq \frac{(5\|f\|^2_{\infty} + 2\sigma^2)\mathrm{vol}_d(P_{\Scal}\Pi)}{n}  \sum_{k=0}^d  \binom{d}{k} \lambda^k \EE[V(W_{\Scal}[k], Z_{\Scal}[d-k])],
\end{align*}
and the conditional variance satisfies
\begin{align}\label{e:cond_var2}
  &\EE[(\bar{f}_{\lambda}^{(1)}(X) - \hat{f}_{\lambda,n,1}(X) )^2| X \in K_{\ee}] \leq \PP(X \in K_{\ee})^{-1} \EE[(\bar{f}_{\lambda}^{(1)}(X) - \hat{f}_{\lambda,n,1}(X) )^2] \nonumber \\
  & \qquad \leq \frac{(5\|f\|^2_{\infty} + 2\sigma^2)\mathrm{vol}_d(P_{\Scal}\Pi)}{n \PP(X \in K_{\ee})}  \sum_{k=0}^d  \binom{d}{k} \lambda^k \EE[V(W_{\Scal}[k], Z_{\Scal}[d-k])].
\end{align}
It remains to control the remaining bias term. By Taylor's theorem, for $f \in \mathcal{C}^{1,\beta}(L)$ with $\beta \in (0,1]$,
\begin{align*}
    |f(z) - f(x) - \nabla f(x)^T(z-x)| &= \left| \int_0^1 [\nabla f(x + t(z-x)) - \nabla f(x)]^T(z-x)\dint t\right| \\
    &\leq \int_0^1 L(t\|z-x\|)^{\beta}\|z-x\| \dint t \leq L\|z-x\|^{1 + \beta}.
\end{align*}
Then, for $x \in \mathrm{supp}(\mu)$,
\begin{align*}
    &|\tilde{f}_{\lambda}(x) - f(x)| = \left| \EE\left[\frac{1}{\mu(Z_x^{\lambda})}\int_{Z_x^{\lambda}} (f(z) - f(x)) \mu(\dint z) \right]\right| \\
    &\leq \left|\EE\left[ \frac{1}{\mu(Z_x^{\lambda})}\int_{Z_x^{\lambda}} \nabla f(x)^T(z - x) \mu(\dint z)\right]\right| \\
    & \quad \quad + \EE\left[ \frac{1}{\mu(Z_x^{\lambda})}\int_{Z_x^{\lambda}} \left|f(z) - f(x) - \nabla f(x)^T(z - x) \right| \mu(\dint z)\right] \\
    &\leq \left|\nabla f(x)^T \int_{\Scal} (z - x) \EE\left[\frac{1_{\{z \in Z_x^{\lambda}\}}}{\mu(Z_x^{\lambda})}\right] \mu(\dint z)\right|  +  \EE\left[\frac{1}{\mu(Z_x^{\lambda})}\int_{\Scal} L\|z-x\|^{1 + \beta}1_{\{z \in Z_x^{\lambda}\}} \mu(\dint z)\right] \\
    &\leq \|\nabla f(x)\| \left\|\int_{\Scal} (z - x) \EE\left[\frac{1_{\{z \in Z_x^{\lambda}\}}}{\mu(Z_x^{\lambda})}\right] \mu(\dint z)\right\| +  L\EE\left[\mathrm{diam}(Z_x^{\lambda} \cap \Scal)^{1 + \beta}\right] \\
    &\leq L\left\| \int_{\Scal} (z - x) \EE\left[\frac{1_{\{z \in Z_x^{\lambda}\}}}{\mu(Z_x^{\lambda})}\right]\mu(\dint z)\right\| +  \frac{Lc_{\beta,P_{\Scal}\Pi}}{\lambda^{1 + \beta}}.
\end{align*}
Up to this point, we have closely followed the proof of Theorem 3 of \cite{mourtada2020minimax} with more general bounds for the parameters of STIT tessellations. For the next step, recall that by the assumptions, $\mu$ has a positive and Lipschitz density $p$ w.r.t. the Lebesgue measure on its support. %By the assumption that $X$ is $s$-sparse, $\mathrm{supp}(\mu) = W \cap \mathcal{S}$, where $\mathcal{S}$ is an $s$-dimensional linear subspace.
To bound the first term above, the proof of Theorem 3 in \cite{mourtada2020minimax} compares the density $F_{\lambda,p}(z) := \EE\left[\frac{p(z)}{\mu(Z_x^{\lambda})}1_{\{z \in Z_x^{\lambda}\}}\right]$ with the density $F_{\lambda, \text{unif}}(z) := \EE\left[\frac{1_{\{z \in Z_x^{\lambda} \cap [0,1]^D\}}}{\mathrm{vol}_D(Z_x^{\lambda}\cap[0,1]^D)}\right]$ (where $p$ is the uniform density on the unit cube). %They then apply their Lemma 1, the proof of which relies heavily on the rectangular geometry of the cells in the Mondrian process. Here, we can avoid the boundary issue in their proof and incorporate the intrinsic dimension of the input with the following modification. 
We instead compare $F_{\lambda, p}$ with the density 
\[ F_{\lambda}(z) := \EE\left[\frac{1_{\{z \in Z_x^{\lambda}\}}}{\mathrm{vol}_d\left(Z_x^{\lambda} \cap  \mathcal{S}\right)}\right], \quad z \in \mathcal{S}.\]
By Lemma \ref{l:int_zero} we obtain the following upper bound on the first term above: 
\begin{align*}
    \left\| \int_{\mathcal{S}} (z - x)  F_{\lambda, p}(z) \dint z\right\| & = \left\|\int_{\mathcal{S}} (z-x)\left(F_{\lambda,p}(z) - F_{\lambda}(z)\right)\dint z \right\| \\
    &\leq 
    \int_{\mathcal{S}} \|z - x\| \left|\EE\left[\frac{p(z)1_{\{z \in Z_x^{\lambda}\}}}{\mu(Z_x^{\lambda})}\right] -  \EE\left[\frac{1_{\{z \in Z_x^{\lambda}\}}}{\mathrm{vol}_d(Z_x^{\lambda} \cap \mathcal{S})}\right] \right| \dint z \\
   % &\leq \int_{\mathcal{S}} \|z - x\| \left|\EE\left[\frac{\mathrm{vol}_d(Z_x^{\lambda} \cap \mathcal{S})p(z)-\mu(Z_x^{\lambda})}{\mu(Z_x^{\lambda})\mathrm{vol}_d(Z_x^{\lambda} \cap \mathcal{S})} 1_{\{z \in Z_x^{\lambda}\}} \right] \right| \dint z \\
    &\leq \int_{\mathcal{S}} \|z - x\| \EE\left[\frac{\int_{Z_x^{\lambda} \cap \Scal}|p(z) - p(y)|\dint y}{\mu(Z_x^{\lambda})\mathrm{vol}_d(Z_x^{\lambda} \cap \mathcal{S})}1_{\{z \in Z_x^{\lambda} \}}\right]\dint z \\
    &\leq \EE\left[\mathrm{diam}(Z_x^{\lambda} \cap \mathcal{S})\frac{\int_{Z_x^{\lambda} \cap \mathcal{S}} \int_{Z_x^{\lambda} \cap \Scal}  |p(z) - p(y)|\dint y\dint z}{\mu(Z_x^{\lambda})\mathrm{vol}_d(Z_x^{\lambda} \cap \mathcal{S})}\right]. 
\end{align*}
Recall from the assumptions that the density $p$ of $\mu$ has a finite Lipschitz constant $C_p > 0$  on its compact and convex $d$-dimensional support $K := \mathrm{supp}(\mu) \subset \Scal$ and we can define $p_0 := \min_{x \in K} p(x) > 0$ and $p_1 := \max_{x \in K} p(x) < \infty$. Also note that the integrand above is zero when $z,y \notin K$. In the following we denote by $K^c := \RR^d \backslash K$ the complement of $K$. Then,
\begin{align*}
    &\left\| \int_{\mathcal{S}} (z - x)  F_{\lambda, p}(z) \dint z\right\| 
    \leq \EE\left[\mathrm{diam}(Z_x^{\lambda} \cap \mathcal{S})\frac{\int_{Z_x^{\lambda} \cap \mathcal{S}} \int_{Z_x^{\lambda} \cap \Scal}  |p(z) - p(y)|\dint y\dint z}{\mu(Z_x^{\lambda})\mathrm{vol}_d(Z_x^{\lambda} \cap \mathcal{S})}\right] \\
    & \qquad \leq \frac{C_p}{p_0}\EE\left[\mathrm{diam}(Z_x^{\lambda} \cap \mathcal{S})\frac{\int_{Z_x^{\lambda} \cap K} \int_{Z_x^{\lambda} \cap K}  \|z - y\|\dint y\dint z}{\mathrm{vol}_d(Z_x^{\lambda} \cap K)\mathrm{vol}_d(Z_x^{\lambda} \cap \mathcal{S})}\right] \\
    & \qquad \qquad + \frac{2C_pp_1}{p_0}\EE\left[\mathrm{diam}(Z_x^{\lambda} \cap \mathcal{S})\frac{\int_{Z_x^{\lambda} \cap K} \int_{Z_x^{\lambda} \cap \Scal \cap K^c}\dint y\dint z}{\mathrm{vol}_d(Z_x^{\lambda} \cap K)\mathrm{vol}_d(Z_x^{\lambda} \cap \mathcal{S})}\right] \\
    & \qquad \leq \frac{C_p}{p_0}\EE\left[\mathrm{diam}(Z_x^{\lambda} \cap \mathcal{S})^2\right] + \frac{2C_pp_1}{p_0}\EE\left[\frac{\mathrm{diam}(Z_x^{\lambda} \cap \mathcal{S})\mathrm{vol}_d\left(Z_x^{\lambda} \cap \Scal \cap K^c\right)}{\mathrm{vol}_d(Z_x^{\lambda} \cap \mathcal{S})}\right] \\
    & \qquad \leq \frac{C_p c_{2, P_{\Scal}\Pi}}{\lambda^{2}p_0} + \frac{2C_pp_1}{p_0}\EE\left[\frac{\mathrm{diam}(Z_x^{\lambda} \cap \mathcal{S})\mathrm{vol}_d\left(Z_x^{\lambda} \cap 
    \Scal \cap K^c\right)}{\mathrm{vol}_d(Z_x^{\lambda} \cap \mathcal{S})}\right],
\end{align*} 
where the last inequality follows from Lemma \ref{l:diameter} and stationarity. 
Now, using stationarity, Fact \ref{f:intersection}, and \eqref{e:zerocell_scaled},
\begin{align*}
\EE\left[\frac{\mathrm{diam}(Z_x^{\lambda} \cap \Scal )\mathrm{vol}_d(Z_x^{\lambda} \cap \Scal \cap K^c)}{\mathrm{vol}_d(Z_x^{\lambda}\cap \Scal)}\right] = \EE\left[\frac{\mathrm{diam}(Z_{0,\Scal})\mathrm{vol}_d(Z_{0, \Scal} \cap \lambda(K^c - x))}{\lambda\mathrm{vol}_d(Z_{0, \Scal})}\right].   
\end{align*}
Thus we have the following upper bound on the conditional bias: 
\begin{align}\label{e:cond_bnd1}
    &\EE[(\tilde{f}_{\lambda}(X) - f(X))^2| X \in K_{\ee}] \leq L^2\EE\left[\left(\left\|\int_{\Scal}(z-X)F_{\lambda,p}(z) \dint z\right\| + \frac{c_{\beta, P_{\Scal}\Pi}}{\lambda^{1 + \beta}}\right)^2 \bigg| X \in K_{\ee}\right] \nonumber \\
    &\leq L^2\EE\left[\left(\frac{c_{\beta, P_{\Scal}\Pi}}{\lambda^{1 + \beta}} + \frac{C_p c_{2, P_{\Scal}\Pi}}{\lambda^{2}p_0} + \frac{2C_p p_1}{\lambda p_0}\EE\left[\frac{\mathrm{diam}(Z_{0,\Scal})\mathrm{vol}_d(Z_{0, \Scal} \cap \lambda(K^c - X))}{\mathrm{vol}_d(Z_{0, \Scal})}\right]\right)^2 \bigg| X \in K_{\ee}\right] \nonumber \\
    &\leq L^2\left(\frac{c_{\beta, P_{\Scal}\Pi}}{\lambda^{1 + \beta}} + \frac{C_p c_{2, P_{\Scal}\Pi}}{\lambda^{2}p_0}\right)^2 
  +  \frac{4C^2_p p^2_1}{\lambda^2 p^2_0}\EE\left[\frac{\mathrm{diam}(Z_{0,\Scal})^2\mathrm{vol}_d(Z_{0, \Scal} \cap \lambda(K^c - X))^2}{\mathrm{vol}_d(Z_{0, \Scal})^2} \bigg| X \in K_{\ee}\right] \nonumber \\
  &\qquad  + \frac{4L^2C_p p_1}{\lambda p_0}\left(\frac{c_{\beta, P_{\Scal}\Pi}}{\lambda^{1 + \beta}} + \frac{C_p c_{2, P_{\Scal}\Pi}}{\lambda^{2}p_0}\right)\EE\left[\frac{\mathrm{diam}(Z_{0,\Scal})\mathrm{vol}_d(Z_{0, \Scal} \cap \lambda(K^c - X))}{\mathrm{vol}_d(Z_{0, \Scal})} \bigg| X \in K_{\ee}\right].
\end{align}
Conditioned on $X \in K_{\ee}$, $\ee B^d \subseteq K - X$. This implies that $\mathrm{vol}_d(Z_{0, \Scal} \cap \lambda(K^c - X)) = 0$ if $\mathrm{diam}(Z_{0, \Scal}) \leq \lambda \ee$. Then, for $k \in \{1, 2\}$,
 \begin{align*}
    &\EE\left[\frac{\mathrm{diam}(Z_{0,\Scal})^k\mathrm{vol}_d(Z_{0, \Scal} \cap \lambda(K^c - X))^k}{\mathrm{vol}_d(Z_{0, \Scal})^k} \bigg| X \in K_{\ee}\right] \\
    &\qquad \qquad \leq \EE\left[\frac{\mathrm{diam}(Z_{0,\Scal})^k1_{\{\mathrm{diam}(Z_{0, \Scal}) \geq \lambda \ee\}}}{\mathrm{vol}_d(Z_{0, \Scal})^2\PP(X \in K_{\ee})} \EE_X\left[\mathrm{vol}_d(Z_{0, \Scal} \cap \lambda(K^c - X))^k\right] \right].
\end{align*}
To bound the inner expectation with respect to $X$, we see that 
\begin{align*}
&\EE_X[\mathrm{vol}_d(Z_{0, \Scal} \cap \lambda(K^c - X))^k] = \int_K p(x) \left(\int_{\RR^d} 1_{\{y \in Z_{0, \Scal} \cap \lambda(K^c - x)\}}\right)^k \dint x \\
& \qquad \leq p_1 \mathrm{vol}_d(Z_{0, \Scal})^{k-1} \int_K \int_{\RR^d} 1_{\{y \in Z_{0, \Scal} \cap \lambda(K^c - x)\}} \dint y \dint x \\
& \qquad =  p_1 \mathrm{vol}_d(Z_{0, \Scal})^{k-1}  \int_{Z_{0, \Scal}} \int_{K} 1_{\{x \in K^c - \frac{y}{\lambda}\}} \dint x \dint y  \\
&\qquad = p_1 \mathrm{vol}_d(Z_{0, \Scal})^{k-1} \int_{Z_{0, \Scal}} \mathrm{vol}_d\left(K \cap K^c - \frac{y}{\lambda}\right) \dint y \\
&\qquad = p_1 \mathrm{vol}_d(Z_{0, \Scal})^{k-1} \int_{Z_{0,\Scal}}\mathrm{vol}_d\left( K\cup K - \frac{y}{\lambda}\right) \dint y - p_1 \mathrm{vol}_d(Z_{0, \Scal})^k \mathrm{vol}_d(K),
\end{align*}
where we have used that $\mathrm{vol}_d(K \cap K^c - y/\lambda) = \mathrm{vol}_d(K) - \mathrm{vol}_d(K \cap K - y/\lambda)$ and $\mathrm{vol}_d(K \cap K - y/\lambda) = 2 \mathrm{vol}_d(K) - \mathrm{vol}_d(K \cup K - y/\lambda)$.
We now observe that the union $K \cup K - \frac{y}{\lambda}$ is a subset of the Minkowski sum $K + \frac{\|y\|}{\lambda}B^d$. By Steiner's formula \cite[Equation (14.5)]{weil}, 
\begin{align*}
\mathrm{vol}_d\left( K\cup K - \frac{y}{\lambda}\right) &\leq \mathrm{vol}_d\left(K + \frac{\|y\|}{\lambda}B^d\right) = \sum_{j=0}^d \left(\frac{\|y\|}{\lambda}\right)^{d-j}\kappa_{d-j}V_j(K) \\
&= \mathrm{vol}_d(K) + \sum_{j=0}^{d-1} \left(\frac{\|y\|}{\lambda}\right)^{d-j}\kappa_{d-j}V_j(K).
\end{align*}
Thus, 
\begin{align*}
\EE_X[\mathrm{vol}_d(Z_{0, \Scal} \cap \lambda(K^c - X))^k] 
&\leq p_1 \mathrm{vol}_d(Z_{0, \Scal})^k \sum_{j=0}^{d-1} \left(\frac{\mathrm{diam}(Z_{0, \Scal})}{\lambda}\right)^{d-j}\kappa_{d-j} V_j(K) \\
&= 2\lambda^{-1} p_1 \mathrm{vol}_d(Z_{0, \Scal})^k \mathrm{diam}(Z_{0, \Scal})V_{d-1}(K) + O(\lambda^{-2}).
\end{align*}
Putting the above bounds together gives
 \begin{align}
    &\EE\left[\frac{\mathrm{diam}(Z_{0,\Scal})^k\mathrm{vol}_d(Z_{0, \Scal} \cap \lambda(K^c - X))^k}{\mathrm{vol}_d(Z_{0, \Scal})^k} \bigg| X \in K_{\ee}\right] \nonumber \\
    &\qquad \qquad \leq \left(\frac{p_1V_{d-1}(K)}{\lambda \PP(X \in K_{\ee})} +O(\lambda^{-2})\right) \EE\left[\mathrm{diam}(Z_{0,\Scal})^{k+1}1_{\{\mathrm{diam}(Z_{0, \Scal}) \geq \lambda \ee\}} \right].
\end{align}
By Holder's inequality, Lemma \ref{l:diameter} and \eqref{e:diamtail},
\begin{align*}
\EE\left[\mathrm{diam}(Z_{0,\Scal})^{k+1}1_{\{\mathrm{diam}(Z_{0, \Scal}) \geq \lambda \ee\}}\right] &\leq \EE\left[\mathrm{diam}(Z_{0,\Scal})^{2k + 2}\right]^{1/2} \PP\left(\mathrm{diam}(Z_{0, \Scal}) \geq \lambda \ee\right)^{1/2} \\
& \leq  c_{k, \ee, P_{\Scal}\Pi} e^{ - \lambda \ee c_{P_{\Scal}\Pi}}.
\end{align*}
Combining the above bounds implies
\begin{align*}\label{e:cond_bnd1}
    &\EE[(\tilde{f}_{\lambda}(X) - f(X))^2| X \in K_{\ee}] \\
    &\leq L^2\left(\frac{c_{\beta, P_{\Scal}\Pi}}{\lambda^{1 + \beta}} + \frac{C_p c_{2, P_{\Scal}\Pi}}{\lambda^{2}p_0}\right)^2 
  +  \frac{4C^2_p p^3_1V_{d-1}(K)}{\lambda^3 p^2_0 \PP(X \in K_{\ee})}c_{2, \ee, P_{\Scal}\Pi} e^{ - \lambda \ee c_{P_{\Scal}\Pi}}\\
  &\qquad  + \frac{4L^2C_p p^2_1}{\lambda^2 p_0}\left(\frac{c_{\beta, P_{\Scal}\Pi}}{\lambda^{1 + \beta}} + \frac{C_p c_{2, P_{\Scal}\Pi}}{\lambda^{2}p_0}\right)\frac{V_{d-1}(K)}{\PP(X \in K_{\ee})}c_{1, \ee, P_{\Scal}\Pi} e^{ - \lambda \ee c_{P_{\Scal}\Pi}} + O(\lambda^{-4}e^{ - \lambda \ee c_{P_{\Scal}\Pi}}).
\end{align*}
Next observe that $\PP(X \in K_{\ee}) \geq \frac{p_0\mathrm{vol}_d(K_{\ee})}{\mathrm{vol}_d(K)} > 0$.  The total conditional risk then satisfies
\begin{align*}
&\EE[(\hat{f}_{\lambda,n, M}(X) - f(X))^2| X \in K_{\ee}] \leq \frac{L^2c_{2,P_{\Scal}\Pi}}{M\lambda^2} + L^2\left(\frac{c_{\beta, P_{\Scal}\Pi}}{\lambda^{1 + \beta}} + \frac{C_p c_{2, P_{\Scal}\Pi}}{\lambda^{2}p_0}\right)^2 \\
  & +  \frac{4C^2_p p^3_1V_{d-1}(K)\mathrm{vol}_d(K)}{\lambda^3 p^2_0 p_0\mathrm{vol}_d(K_{\ee})}c_{2, \ee, P_{\Scal}\Pi} e^{ - \lambda \ee c_{P_{\Scal}\Pi}}\\
  & + \frac{4L^2C_p p^2_1}{\lambda^2 p_0}\left(\frac{c_{\beta, P_{\Scal}\Pi}}{\lambda^{1 + \beta}} + \frac{C_p c_{2, P_{\Scal}\Pi}}{\lambda^{2}p_0}\right)\frac{V_{d-1}(K)\mathrm{vol}_d(K)}{p_0\mathrm{vol}_d(K_{\ee})}c_{1, \ee, P_{\Scal}\Pi} e^{ - \lambda \ee c_{P_{\Scal}\Pi}} + O\left(\lambda^{-4}e^{ - \lambda \ee c_{P_{\Scal}\Pi}}\right) \\
  &  + \frac{(5\|f\|^2_{\infty} + 2\sigma^2)\mathrm{vol}_d(P_{\Scal}\Pi)\mathrm{vol}_d(K)}{n p_0 \mathrm{vol}_d(K_{\ee})}  \sum_{k=0}^d  \binom{d}{k} \lambda^k \EE[V(W_{\Scal}[k], Z_{\Scal}[d-k])].
\end{align*}
Extracting leading order terms, we have
\begin{align*}
\EE[\hat{f}_{\lambda,n, M}(X) - f(X))^2 | X \in K_{\ee}] \leq O\left( \frac{L^2}{\lambda^2M} + \frac{L^2}{\lambda^{2(1+ \beta)}} + \frac{\lambda^d}{n}\right).
\end{align*}
Finally, letting $\lambda = \lambda_n \sim L^{2/(d+2\beta + 2)}n^{1/(d+2\beta+2)}$ and $M = M_n \gtrsim \lambda_n^{2\beta}$ as $n \to \infty$ gives the rate $O\left(L^{2d/(d + 2\beta + 2)}n^{-(2\beta + 2)/(d + 2\beta + 2)}\right)$ in Corollary \ref{cor:rate2}.
In the unconditional case,
\begin{align*}
\EE[\hat{f}_{\lambda,n, M}(X) - f(X))^2] \leq O\left(\frac{L^2}{\lambda^2 M} + \frac{L^2}{\lambda^{\min\{3, 2(1 + \beta)\}}} + \frac{\lambda^d}{n}\right).
\end{align*}
This bound gives the same rate as above when $\beta \leq 1/2$. 
When $\beta > 1/2$, this bound gives the suboptimal rate $O\left(L^{2d/(d+3)}n^{-3/(d+3)}\right)$ by letting $\lambda = \lambda_n \sim L^{2/(d + 3)}n^{1/(d+3)}$ and $M = M_n \gtrsim \lambda_n$ as $n \to \infty$.
\end{proof}

\subsection{Proof of Theorem \ref{t:poisson}}

By Corollary 1 of \cite{Thale2013Poisson}, the typical cell of a STIT tessellation with lifetime parameter $\lambda$ has the same distribution as the typical cell of a Poisson hyperplane tessellation with intensity $\lambda$ and the same normalized associated zonoid, or equivalently, the same directional distribution. The distribution of the typical cell determines the distribution of the zero cell \cite[Theorem 10.4.1]{weil}, and thus the same proof methods used in Theorems \ref{t:rate1} and \ref{t:rate2} can be applied in this setting and the results follow.

%%%%%%%%%%%%%%%%%%%%%%%%%%%%%%%%%%%%%%%%%%%%%%%%%%%%%%%%%%%%%%

\section{Discussion and Future Work}\label{s:discussion}

This work expands and strengthens the theoretical basis for purely random forests, and establishes stochastic geometry as a promising toolkit for analyzing regression and classification algorithms based on random partitions. In particular, we showed that a large class of random forests built from stationary hyperplane partitions with oblique splits all achieve the same minimax optimal rates as Mondrian forests proved in \cite{mourtada2020minimax}. We also extended these rates to depend on a notion of the intrinsic dimension of the input as opposed to the ambient dimension of the feature space. This work motivates many more questions at the intersection of stochastic geometry and machine learning. We outline a few future research directions here. 

First, the definition of low intrinsic dimension used in our assumptions has limited applicability.  However, we hope that our results and proof techniques can form a basis for future work to obtain optimal rates under more general notions of low dimensionality of the input. 
Additionally, as mentioned in Remark \ref{Rem:adaptivity}, to obtain these rates in practice one needs an adaptive way of tuning the lifetime parameter and number of trees, since the intrinsic dimension and regularity are not known {\it a priori}. For adaptation to regularity, a model aggregation method was proposed by \cite{mourtada2020minimax} for Mondrian forests which could potentially be extended to STIT forests.

Another open question is whether the flexibility of the directional distribution allows us to find ``optimal" split directions, or directional distribution $\phi$, for a given data set. For example, a directional distribution that depends on the covariate distribution may improve performance and decrease computational costs by decreasing the complexity of the partition needed to achieve optimal rates. The flexibility of these models also motivates a study of whether, under different assumptions about the underlying function $f$, one can obtain convergence rates that depend on the directional distribution, and show optimal rates are achieved with a good choice of this parameter. In particular, we might expect that with appropriate choice of directional distribution, STIT random forests will adapt to the same types of low dimensional structure that other purely random forest variants have been shown to adapt to under a modification of the split direction probabilities. For example, centered random forests obtain convergence rates that depend on the sparsity level of the regression function described by number of relevant features \cite{Biau2012, Klusowski2021}.

A third research direction concerns using the theory of stochastic geometry to study random forests built from random tessellations where split locations depend on the data set. For example, we may consider STIT or Poisson hyperplane tessellations associated with a non-stationary intensity measure and 
somehow incorporate the given data set into this measure to improve the performance of this class of random forests. In the stochastic geometry literature, some non-stationary random tessellation models have been studied.
Sections 11.3 and 11.4 in \cite{weil} collect results on non-stationary flat processes and Poisson hyperplane tessellations, many of which are due to \cite{Schneider2003}. Also of note is work by \cite{Hoffman2007} who studied a generalization of the associated zonoid for non-stationary Poisson hyperplane tessellations.

\bibliographystyle{plain}
\bibliography{biblio}

\end{document}